\documentclass[lettersize,journal]{IEEEtran}
\usepackage{amsmath,amsfonts}
\usepackage{algorithm}
\usepackage{algorithmic}
\usepackage{array}
\usepackage[caption=false,font=normalsize,labelfont=sf,textfont=sf]{subfig}
\usepackage{textcomp}
\usepackage{stfloats}
\usepackage{url}
\usepackage{verbatim}
\usepackage{graphicx}
\usepackage{cite}
\usepackage{bm}
\usepackage{amsthm}
\usepackage{setspace}
\newtheorem{example}{Example}
\newtheorem{thm}{Theorem}
\newtheorem{remark}{Remark}
\newtheorem{defn}{Definition}

\begin{document}

\title{Graph Linear Canonical Transform: Definition, Vertex-Frequency Analysis and Filter Design}

\author{Jian Yi Chen and Bing Zhao Li
\thanks{J. Y. Chen is with the School of Mathematics and Statistics, Beijing Institute of Technology
	, Beijing 100081, China. (e-mail: 15030218436@163.com)}
\thanks{B. Z. Li is with the Beijing Key Laboratory on MCAACI, Beijing Institute of Technology, Beijing 100081, China. (e-mail: li\_bingzhao@bit.edu.cn)}}

\markboth{Journal of \LaTeX\ Class Files,~Vol.~14, No.~8, August~2021}%
{Shell \MakeLowercase{\textit{et al.}}: A Sample Article Using IEEEtran.cls for IEEE Journals}

\IEEEpubid{0000--0000/00\$00.00~\copyright~2021 IEEE}

\maketitle

\begin{abstract}
This paper proposes a graph linear canonical transform (GLCT) by decomposing the linear canonical parameter matrix into fractional Fourier transform, scale transform, and chirp modulation for graph signal processing. The GLCT enables adjustable smoothing modes, enhancing alignment with graph signals. Leveraging traditional fractional domain time-frequency analysis, we investigate vertex-frequency analysis in the graph linear canonical domain, aiming to overcome limitations in capturing local information. Filter design methods, including optimal design and learning with stochastic gradient descent, are analyzed and applied to image classification tasks. The proposed GLCT and vertex-frequency analysis present innovative approaches to signal processing challenges, with potential applications in various fields.
\end{abstract}


\section{Introduction} \label{sec1}
Graph signal processing has a broad spectrum of applications in computer vision, image processing, and social network analysis. Advancements in these technologies will catalyze technological innovation and societal progress, offering abundant developmental opportunities and benefits for humanity \cite{sandryhaila2014big, sandryhaila2013discrete, ortega2018graph}.
The essence of graph signal processing lies in the construction of dictionary atoms, followed by the representation of graph signals as linear combinations of these atoms \cite{shuman2020localized}. This representation technique finds diverse applications, including data visualization analysis, statistical analysis, compression, and regularization in machine learning, and ill-posed inverse problems such as repair, denoising, and classification tasks. 

When the atom is the eigenvector (Fourier basis) obtained from the Laplacian matrix decomposition of the graph, different Fourier bases correspond to distinct levels of smoothness. Their associated eigenvalues provide a gradient characterization of the smoothness of the graph signal, effectively functioning as frequencies \cite{shuman2013emmerging, sandryhaila2014discrete}. Lower eigenvalues indicate slower changes in the corresponding Fourier basis, resulting in consistent signal values across similar nodes. Conversely, higher eigenvalues signify more rapid changes in the Fourier basis, leading to inconsistent signal values across similar nodes. Thus, the coefficients of the graph Fourier transform (GFT) represent the amplitude of the graph signal on the corresponding frequency component, reflecting the strength of the signal on that frequency component. 

Based on the GFT, various graph analysis methods have emerged, with the core focus being on designing different dictionary atoms to represent graph signals. The effectiveness of signal analysis refers to the ability to capture important information in a signal with fewer mathematical descriptions \cite{rubinstein2010dictionaries}, which essentially requires atoms to be able to sparsely represent the signal. 
Therefore, the effectiveness of vertex-frequency analysis of graph signals requires atoms to have the following characteristics in the spatial domain:
\begin{itemize}
	\item Interpretability: The methods should enable the explanation of the underlying graph structure.
	\item Multiresolution: They should be able to continuously approximate signals from low-frequency resolution to high-frequency resolution.
	\item Locality: Emphasis should be placed on the local structure of the graph.
	\item Multiscale: They should encompass vertex-frequency distributions of various scales.
	\item Smooth diversity: The ability to select different smoothing modes should be included.
\end{itemize}
Traditional vertex-frequency analysis methods can effectively fulfill the first four characteristics but are inadequate in addressing the last one\cite{stankovic2017vertex, stankovic2020vertex, shuman2012windowed}. The fixed smoothing mode, based on GFT, solely relies on its Laplacian matrix. Consequently, traditional vertex-frequency analysis methods are confined to characterizing signals from the horizontal and vertical directions of the vertex-frequency plane. This limitation stems from the mathematical structure of their atoms, typically comprising time shift, frequency shift, or scaling operations. As a result, these methods can only impart changes to atoms in the horizontal and vertical directions of the vertex-frequency plane\cite{shuman2016vertex, stankovic2019vertex}. Consequently, traditional vertex-frequency analysis methods fail to fully exploit the inherent geometric characteristics of the signal itself and encounter bottlenecks in enhancing signal processing performance further. An effective solution entails the introduction of a new vertex-frequency analysis method.

The Linear canonical transform (LCT) is a more generalized integral transform, with the Fourier transform, fractional Fourier transform (FrFT), and Fresnel transform all being regarded as special forms of LCT \cite{pei2002, Healy2016, Pei2001}. It demonstrates strong flexibility due to its additional parameters. Building upon the advantages of this transform method, it has been further integrated into graph signal processing technology, leading to significant research achievements. These include the development of graph fractional Fourier transform \cite{wang2017fractional}, graph linear canonical transform (GLCT) \cite{zhang2023discrete} based on adjacency matrix , and related vertex-frequency analysis methods \cite{wu2020fractional, yan2021windowed}.
The linear canonical parameter matrix can be decomposed into a linear combination of FrFT, scale transform, and chirp modulation (CM). Leveraging this pivotal property, this paper proposes a GLCT based on the graph Laplacian matrix, which significantly diverges from previous work \cite{zhang2023discrete}. The GLCT facilitates the adjustment of the smoothing mode to better align with the graph signal. Additionally, building upon the research foundation of traditional fractional domain time-frequency analysis methods \cite{brown2009general,wei2022linear,wei2014generalized,kou2012windowed}, this paper explores vertex-frequency analysis methods in the graph linear canonical domain. This research endeavors to address the inherent limitations of graph signal processing techniques in characterizing local information.

The vertex-frequency analysis method based on GLCT enables atoms to select different smoothing modes, facilitating optimal matching of intrinsic local features of the signal. This approach offers new ideas and methods for addressing problems in signal processing. 
This paper first presents a systematic framework for vertex-frequency analysis based on GLCT. It then discusses three definitions of filters:
(1) The filter is constructed based on the convolution and translation operators of GLCT, leading to the definition of windowed GLCT.
(2) By utilizing the spectral shift form of the kernel of the bandpass function in the linear canonical domain of the graph, we can define GLCT-based wavelet transform and S-transform.
(3) A filter constructed based on the vertex neighborhood can define local GLCT.
In practical applications, selecting the appropriate filter is crucial. For instance, various spectral neural networks primarily rely on different filter learning methods \cite{chen2020understanding, bo2023survey, wang2022powerful, ozturk2021optimal}. 
Therefore, we analyze methods for filter design, discussing optimal filter design as well as filter learning based on stochastic gradient descent methods. Subsequently, we apply the designed filters to image classification tasks. The filter design proposed in this paper can be utilized as convolutional layers in spectral graph neural networks \cite{defferrard2016convolutional, chien2020adaptive, bruna2013spectral, liao2019lanczosnet}, showcasing significant potential for various applications.

This paper is structured through the following steps. Section \ref{sec2} presents the basic theory of graph signal processing and defines time-frequency analysis methods based on LCT. In Section \ref{sec3}, the GLCT is proposed through FrFT, scaling, and CM. In Section \ref{sec4}, the definition of the vertex-frequency analysis method based on GLCT is discussed and summarized. In Section \ref{sec5}, we discuss optimal filter design and filter learning, and we apply the designed filters to image classification tasks.
Finally, Section \ref{sec6} concludes this paper.

\section{Preliminaries}
\label{sec2}
\subsection{The graph signal}
Consider a weighted, undirected, and connected graph $\mathcal{G}=\{\mathcal{V}, \mathcal{E}, \mathrm{W}\}$, where $\mathcal{V}=\{v_0,v_1,\cdots,v_{N-1}\}$ with $|\mathcal{V}|=N$ represents a finite set of vertices, $\mathcal{E}$ denotes the set of edges, and $\mathrm{W}$ represents the weighted adjacency matrix. In this graph, if there exists an edge $e=(i,j)$ connecting vertices $i$ and $j$, the weight of the edge is denoted by $w_{i,j}$; otherwise, $w_{i,j}=0$ \cite{shuman2013emmerging}.

The non-normalized graph Laplacian is a symmetric difference operator $\mathcal{L}=\mathbf{D}-\mathbf{W}$, where $\mathbf{D}$ represents a diagonal degree matrix of the graph $\mathcal{G}$. The $i$-th diagonal element $d_i$ of $\mathbf{D}$ is equal to the sum of the weights of all the edges incident to vertex $i$. 
Since the matrix $\mathcal{L}$ is real and symmetric, it possesses a complete set of orthonormal eigenvectors denoted as $\{\mathbf{u}_k\}_{k=0,1,\cdots,N-1}$. These eigenvectors are sorted in ascending order according to their corresponding eigenvalues $0=\widetilde{\lambda}_0<\widetilde{\lambda}_1\leq \widetilde{\lambda}_2 \leq \cdots \leq \widetilde{\lambda}_{N-1}$.
Thus, using the unitary column matrix $\mathbf{U}=\left[\mathbf{u}_0,\mathbf{u}_1,\cdots,\mathbf{u}_{N-1}\right]$ and the diagonal matrix $\boldsymbol{\Lambda}=\mathrm{diag}\left(\left[\widetilde{\lambda}_0, \widetilde{\lambda}_1, \cdots, \widetilde{\lambda}_{N-1}\right]\right)$, we can express 
\begin{align}
	\mathcal{L}=\mathbf{U} \boldsymbol{\Lambda} \mathbf{U}^H
\end{align}
where the superscript $H$ represents the Hermitian transpose operation. 
While we use the non-normalized graph Laplacian $\mathcal{L}$ for explanation in this article, we can still use the normalized graph Laplacian $\mathcal{L}_{\text{norm}} = \mathbf{E}_N - \mathbf{D}^{-1/2} \mathbf{W} \mathbf{D}^{-1/2}$, where $\mathbf{E}_N$ represents an N-dimensional identity matrix\cite{shuman2020localized}.

A signal defined on the graph $\mathcal{G}$ can be represented as a mapping \cite{yan2021windowed}
\begin{align}
	{f}: \mathcal{V} & \rightarrow \mathbb{R} \nonumber \\
	v_i & \mapsto f(i). 
\end{align}
Alternatively, we can express $\mathbf{f}$ as a real-valued vector $\mathbf{f}=\left[f(0), f(1),\cdots,f(N-1)\right]^T \in \mathbb{R}^N$, where $f(i)$ represents the value associated with the $i$-th vertex.
The Laplacian matrix is an operator that captures the local smoothness of a graph signal. When it acts on the graph signal $f$, it is defined as
\begin{align}
	(\mathcal{L} \mathbf{f})(i)=\sum_{j \in \mathcal{N}i} w_{i, j}[f(i)-f(j)]
\end{align}
where the neighborhood $\mathcal{N}_i$ represents the set of vertices connected to vertex $i$ by an edge. This operator quantifies the difference in signal between the central node and its neighboring nodes.

\begin{example}
	\label{exam1}
	We simulate two distinct graphs: a 128-node sensor network graph (Fig. \ref{fig1}. (a)) and a 256-node Swiss roll graph (Fig. \ref{fig1}. (b)). We consider the random signal for both cases, as shown in Fig. \ref{fig1} (c) and (d), respectively.
\end{example}

\subsection{The graph fractional Fourier transform}
\subsubsection{Graph Fourier transform}
For any function $\mathbf{f}\in \mathbb{R}^N$ defined on the vertices of $\mathcal{G}$, its Graph Fourier Transform (GFT) is defined as  \cite{shuman2013emmerging}
\begin{align}
	\label{eq3}
	\hat{f}(k) = \langle f, u_k\rangle = \sum_{n=0}^{N-1} f(n)u_k^*(n), \quad k=0,1,\cdots,N-1.
\end{align}
The inverse GFT is given by
\begin{align}
	f(n) = \sum_{k=0}^{N-1} \hat{f}(k)u_k(n), \quad n=0,1,\cdots,N-1.
\end{align}
Therefore, we can observe that the GFT can be represented using the matrix $\mathbf{U}$, such that $\hat{\mathbf{f}} = \mathcal{F}\mathbf{f} = \mathbf{U}^H \mathbf{f}$ \cite{ortega2018graph}. In the case where the GFT matrix $\mathcal{F}$ is ortho-diagonalized, it can be expressed as $\mathcal{F}=\mathbf{U}^H=\mathbf{Q} \mathbf{R} \mathbf{Q}^H$, where the matrices $\mathbf{Q}=\left[\mathbf{q}_0, \mathbf{q}_1, \cdots, \mathbf{q}_{N-1}\right]$ and $\mathbf{R}=\mathrm{diag}\left( \left[r_0, r_1, \cdots, r_{N-1}\right] \right)$ are obtained through the spectral decomposition of the GFT matrix.

\subsubsection{Graph fractional Fourier transform} \label{sec2.2.2}
By denoting $\mathbf{J}=\mathrm{diag}\left( \left[\zeta_0, \zeta_1, \cdots, \zeta_{N-1}\right] \right)=R^{\alpha}$, where $\zeta_l=r_l^{\alpha}$, the graph fractional Fourier transform (GFrFT) is defined as $\hat{\mathbf{f}}_{\alpha} = \mathcal{F}_{\alpha}\mathbf{f} = \mathbf{Q} \mathbf{J} \mathbf{Q}^H \mathbf{f}$ \cite{wang2017fractional}. There is \cite{yan2021windowed}
\begin{align}
	\hat{f}_{\alpha}(k) 
	= \sum_{n=0}^{N-1}\sum_{l=0}^{N-1}  f(n) q_l (k) \zeta_l q_l^*(n) 
\end{align}
where the papameter  $\alpha$ is the order of the fractional Fourier transform, and $k=0,1,\cdots,N-1$.
Similar to (\ref{eq3}), we define the graph fractional Laplacian operator $\mathcal{L}_{\alpha}$ as follows
\begin{align}
	\label{eq8}
	\mathcal{L}_{\alpha}=\mathcal{F}_{\alpha}^H \boldsymbol{\Lambda} \mathcal{F}_{\alpha}.
\end{align}
The inverse GFrFT is given by
\begin{align}
	f(n) = \sum_{k=0}^{N-1}\hat{f}_{\alpha}(k) q_l^* (k) \zeta_l q_l(n) , \quad n=0,1,\cdots,N-1.
\end{align}

\subsection{Linear canonical transform}
The linear canonical transform (LCT) of the signal $f(t) \in L^2(\mathbb{R})$ with the parameters $A=(a,b,c,d)$, where $ad-bc=1$, is defined as follows \cite{pei2002, Healy2016, Pei2001}
\begin{align} \label{eq1}
	& L^{A}\{f\}(u) 
	= \begin{cases}  \int_{-\infty}^{\infty} f(t) K_{A}(t, u)\mathrm{d} t & b \neq 0 \\
		\sqrt{d}  e^{ \mathbf{i} \left(cd/2 \right)u^2} f(dt) & b=0\end{cases}  \\
	& K_{A}(t, u)= \sqrt{\frac{1}{\mathbf{i}2\pi b}} e^{ \mathbf{i} \left(\frac{d}{b} u^2-\frac{2}{b} u t+\frac{a}{b} t^2\right) }
\end{align}
where $\mathbf{i}$ reprents an imaginary unit.
When the parameter $\mathrm{A}=({\cos}\alpha, {\sin}\alpha, -{\sin}\alpha, {\cos}\alpha)$ is used, the LCT becomes the $\alpha$-th order fractional Fourier transform (FrFT)
\begin{align}
	&L^{({\cos}\alpha, {\sin}\alpha, -{\sin}\alpha, {\cos}\alpha)}\{f\}(u)=\int_{-\infty}^{\infty} K_{\alpha}(t, u) f(t) \mathrm{d} t \nonumber \\
	&
	K_{\alpha}(t, u)= 
	\left\{\begin{array}{l}
		A_\alpha \exp \left\{\mathbf{i} \frac{u^2\cot \alpha}{2} -\mathbf{i} t u \csc \alpha+\mathbf{i} \frac{t^2\cot \alpha }{2} \right\} \\
		\alpha \neq n \pi \\
		\delta(u-t), \alpha=2 n \pi \\
		\delta(u+t), \alpha=2 n \pi-\pi
	\end{array}\right.
	\nonumber \\
	& A_\alpha=\sqrt{\frac{1-\mathbf{i}\cot (\alpha)}{2 \pi}}, 0<|\alpha|<\pi .
\end{align}

The LCT with parameter $A=(a,b,c,d)$ can be decomposed into a combination of various specialized transforms, achieved through the decomposition of the parameter matrix. This article will utilize the following decomposition
\begin{align} \label{eq12}
	\left[\begin{array}{ll}
		a & b \\
		c & d
	\end{array}\right]=\left[\begin{array}{ll}
		1 & 0 \\
		\xi & 1
	\end{array}\right]\left[\begin{array}{cc}
		\sigma & 0 \\
		0 & \sigma^{-1}
	\end{array}\right]\left[\begin{array}{cc}
		\cos \beta & \sin \beta \\
		-\sin \beta & \cos \beta
	\end{array}\right]
\end{align}
where $A$ is decomposed into FrFT, scaling, and chirp modulation (CM), and $\xi$ is the CM parameter, $\sigma$ denotes the scaling parameter, $\beta$ is the FrFT parameter. The parameter relations between $(a,b,c,d)$ and $(\xi,\sigma,\beta)$ are
$\xi=\frac{a c+b d}{a^2+b^2}, \sigma=\sqrt{a^2+b^2}, \beta=\cos ^{-1}\left(\frac{a}{\sigma}\right)=\sin ^{-1}\left(\frac{b}{\sigma}\right)$.
In Section \ref{sec3}, we will present the definition of graph linear canonical transform (GLCT) based on this parameter matrix factorization.

\section{Graph linear canonical transform} \label{sec3}
Building upon the centrality and scalability of the LCT eigendecomposition approach, Zhang proposes the definition of the GLCT based on the adjacency matrix by integrating the graph chirp-Fourier transform, graph scale transform, and graph fractional Fourier transform \cite{zhang2023discrete}. Similarly, this article proposes a GLCT based on the Laplacian matrix through decomposition of the LCT parameter matrix in (\ref{eq12}). However, there is a significant difference between the definition in this article and the definition in \cite{zhang2023discrete}.

Since any LCT matrix is non-commutative, it should be implemented in the order of FrFT, scaling, CM. To derive the definition of GLCT based on the Laplacian matrix, we accomplish it through the following three stages.
\begin{description}
	\item[\textbf{Step1}] \textbf{GFrFT stage} \\
	For a graph signal, initiate the GFrFT as delineated in Section \ref{sec2.2.2}, expressed as 
	\begin{align} \label{eq13}
		\hat{\mathbf{f}}_{\beta} = \mathcal{F}_{\beta}\mathbf{f} = \mathbf{Q} \mathcal{K} \mathbf{Q}^H \mathbf{f}
	\end{align}
	where $\mathcal{K}=\mathrm{diag}\left( \left[\kappa_0, \kappa_1, \cdots, \kappa_{N-1}\right] \right)=R^{\beta}$, $\kappa_l=r_l^{\beta}, l=0,\cdots,N-1$, and $\beta=\cos ^{-1}\left(\frac{a}{\sigma}\right)=\sin ^{-1}\left(\frac{b}{\sigma}\right)$.
	\item[\textbf{Step2}] \textbf{Graph scaling stage} \\
	In graph signal processing, defining scale transform differs from traditional signal processing methods. Hammond introduced the concept of scale on the graph by utilizing the scale property of the Fourier transform \cite{hammond2011wavelets}, defining the scale operators at scale $s$ as $S_f^s=f(s\mathcal{L})$.
	Therefore, this article considers the scale Laplacian matrix $\mathbf{S}={\sigma} \mathcal{L}$. \\
	Pei noted that the scaled signal can be derived by decomposing the signal into orthogonal Hermite eigenspaces and subsequently reconstructing the signal using scale Hermite eigenvectors with identical inner product coefficients \cite{pei2011discrete}. We employ this concept to define the graph scale transform. \\
	The graph spectrum decomposition of $\mathbf{S}$ is represented as 
	$\mathbf{S}=\mathbf{Y} \mathbf{X} \mathbf{Y}^H$, and
	$\mathbf{Y}^H=\mathbf{P} {\Xi} \mathbf{P}^H$, where $\mathbf{P}=\left[\mathbf{p}_0, \mathbf{p}_1, \cdots, \mathbf{p}_{N-1}\right]$ and $ {\Xi}=\mathrm{diag}\left( \left[\iota_0, \iota_1, \cdots, \iota_{N-1}\right] \right)$. 
	The graph scale transform is defined as 
	\begin{align}\label{eq14}
		{ST}_{\sigma}\mathbf{f}=\mathbf{P}\mathbf{Q}^H\mathbf{f}
	\end{align} 
	where $\sigma=\sqrt{a^2+b^2}$.
	\item[\textbf{Step3}] \textbf{Graph CM stage} \\
	In classical signal processing, the modulation operator is defined by $M_{\xi}f=e^{\mathbf{i}\pi \xi t}f(t)$. Shuman defined the modulation operator on the graph as $M_{k}f={f(n)} {u}_{k}^2(n)$, where $k \in \{0,\cdots,N-1\}$, recognizing that the essence of the modulation operator lies in multiplying by the Laplacian eigenvectors. \\
	However, the scenario addressed in this article pertains to CM, where $CM_{\xi}f=e^{\mathbf{i}\pi \xi t^2}f(t)$. The concept of CM on the graph has not yet been clearly defined, and defining it from a practical perspective is quite challenging. Hence, we define it based on the chirp signal spectrum perspective.
	
	In traditional signal processing, the Fourier spectrum of a chirp signal $s(t)=\operatorname{rect}\left(\frac{t}{T}\right) e^{j \pi \xi t^2}$ can be approximated as $F\{s\}(u) \approx \operatorname{rect}\left(\frac{u}{|\xi| T}\right) e^{-j \pi \frac{u^2}{\xi}}$ \cite{cumming2005digital}. 
	Similarly, we define a signal in a graph spectral domain as 
	\begin{align}
		&\hat{s}_{\xi}(k)= 
		\left \{
		\begin{array}{ll}
			e^{j \frac{\widetilde{\lambda}_k^2}{\xi}},                    & \xi \ne 0\\
			\mathbf{1}_N,    & \xi=0
		\end{array}
		\right.
	\end{align}
	where $\mathbf{1}_N$ represents an N-dimensional column vector with all elements being 1. 
	Next, the inverse GFT is performed to obtain its vertex domain form as $s(n) = \mathbf{U} \hat{s}_{\xi}(k)$.
	We utilize this signal as the CM term for the graph signal. Let 
	\begin{align}
		\Upsilon = \mathrm{diag} \left(\mathbf{U} \hat{s}_{\xi}(k)\right)
	\end{align}
	thus, defining the CM operator on the graph as 
	\begin{align} \label{eq15}
		CM_{\xi}\mathbf{f}= \Upsilon\mathbf{f}
	\end{align}
	where $\xi=\frac{a c+b d}{a^2+b^2}$.
\end{description}

Based on the aforementioned three stages, we define the GLCT by performing FrFT in (\ref{eq13}), scaling in (\ref{eq14}), and CM in (\ref{eq15}) on a graph signal $\mathbf{f}$.

\begin{defn}
	For the parameter $A=(a,b,c,d)$ of LCT, denote $\sigma=\sqrt{a^2+b^2}$, $\xi=\frac{ac+bd}{\sigma^2}$, $\beta=\cos^{-1}\left(\frac{a}{\sigma}\right)=\sin^{-1}\left(\frac{b}{\sigma}\right)$.
	The graph linear canonical transform (GLCT)  is defined by
	\begin{align}
		\hat{\mathbf{f}}_{A} = \mathcal{F}_{A}\mathbf{f} 
		= &\Upsilon (\mathbf{P}\mathbf{Q}^H)(\mathbf{Q} \mathcal{K} \mathbf{Q}^H) \mathbf{f} 
		=  \Upsilon \mathbf{P} \mathcal{K} \mathbf{Q}^H \mathbf{f}.
	\end{align}
	Its algebraic form is
	\begin{align}
		\hat{f}_{A}(k) 
		= \sum_{n=0}^{N-1} \sum_{l=0}^{N-1} f(n)  \epsilon_k p_l(k) \kappa_l q_l^*(n).
	\end{align}
	The corresponding eigenvalues are
	$\boldsymbol{\Lambda}_A
		=\mathrm{diag}\left(\left[{\lambda}_0, {\lambda}_1, \cdots, {\lambda}_{N-1}\right]\right)$,
	where ${\lambda}_l=\sigma \widetilde{\lambda}_l^{\beta}, l=0,\cdots,N-1$.
\end{defn}

Likewise, when $\mathrm{A}=({\cos}\alpha, {\sin}\alpha, -{\sin}\alpha, {\cos}\alpha)$ is used, the equation $\hat{f}_{A=({\cos}\alpha, {\sin}\alpha, -{\sin}\alpha, {\cos}\alpha)}(k)= \hat{f}_{\alpha}(k)$ holds true.
We can also define graph linear canonical Laplacian operator as
\begin{align}
	\mathcal{L}_{A}=\mathcal{F}_{A}^H \boldsymbol{\Lambda}_A \mathcal{F}_{A}.
\end{align}

The inverse GLCT (IGLCT) is defined as $\mathbf{f}=\mathcal{F}^{H}_A \hat{\mathbf{f}}_A= \mathbf{Q} \mathcal{K}^{H} \mathbf{P}^H \Upsilon^{H}\hat{\mathbf{f}}_A$, which is equal to
\begin{align}
	f(n) 
	= \sum_{k=0}^{N-1} \sum_{l=0}^{N-1} \hat{f}_{A}(k) 
	q_l(n) \kappa_l^{*} p_l^*(k) \epsilon_k^{*}.
\end{align}

\begin{example}
	Perform the GLCT with parameters $\sigma=1/3$, $\xi=0$, and $\beta=0.8$ on the graph and graph signal in Example \ref{exam1}. The results are depicted in Fig. \ref{fig1} (e) and (f), respectively.
\end{example}

\begin{example} \label{exam6}
	To further elucidate the role of GLCT, we visualize its feature vectors. Consider a K-nearest neighbors (K-NN) graph with 100 vertices, where each vertex connects to 4 edges. Figures \ref{fig1_2} (a-c) depict the 9th, 49th, and 89th eigenvectors of GFT. Figures \ref{fig1_2} (d-f) illustrate the 9th, 49th, and 89th eigenvectors of GLCT under parameters $\sigma=1$, $\xi=0.8$, and $\beta=0.95$. Figures \ref{fig1_2} (g-i) illustrate the 9th, 49th, and 89th eigenvectors of GLCT under parameters $\sigma=0.4$, $\xi=0$, and $\beta=0.9$.
	It can be observed that as the eigenvalues increase, the eigenvectors become less smooth. The GLCT can modify the smoothing mode by adjusting parameters, which enables a better fit for various graph signals.
\end{example}

\begin{figure}
	\centering
	\subfloat[]{  
		\includegraphics[width=1.8 in]{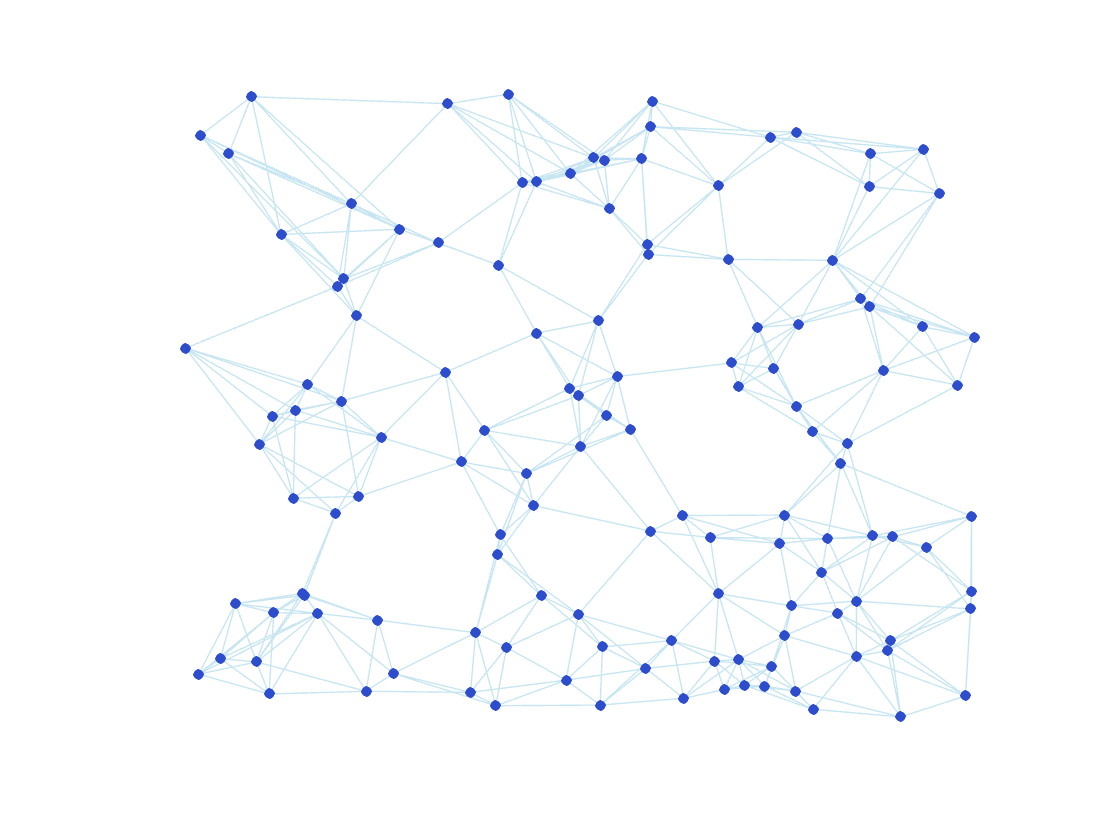}}
	\subfloat[]{
		\includegraphics[width=1.8 in]{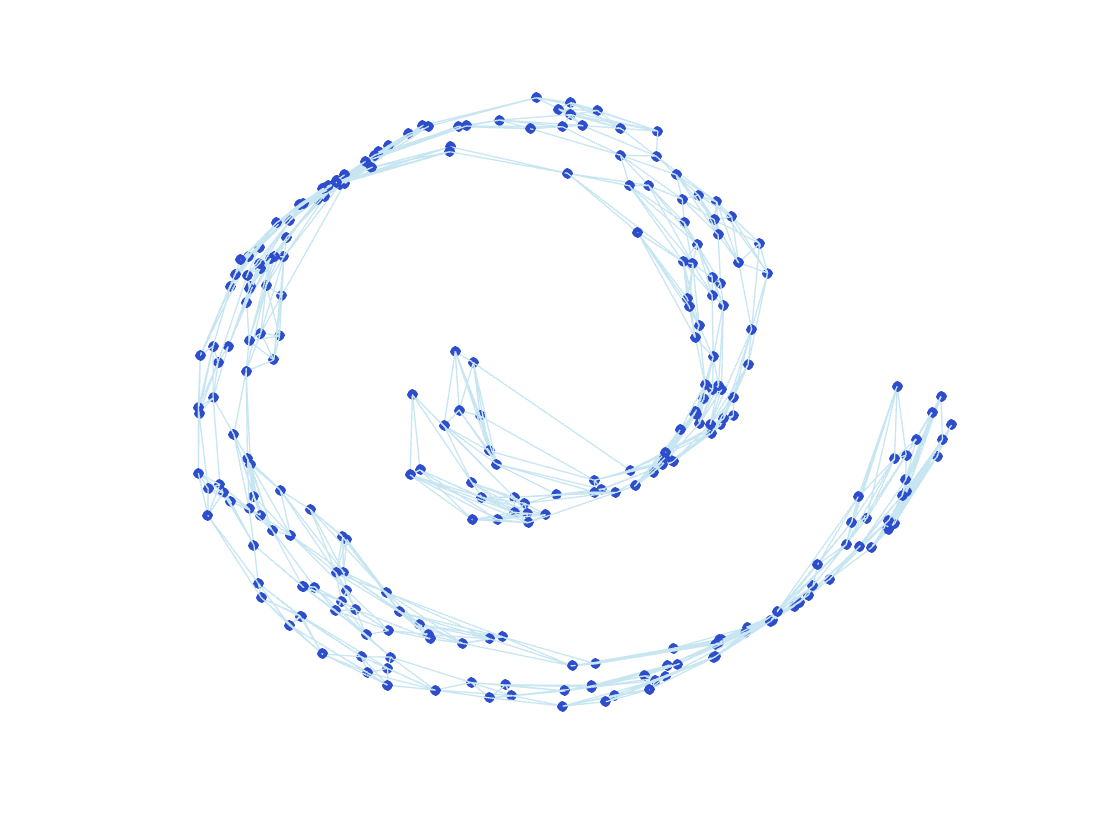}}
	\quad
	\subfloat[]{
		\includegraphics[width=1.8 in]{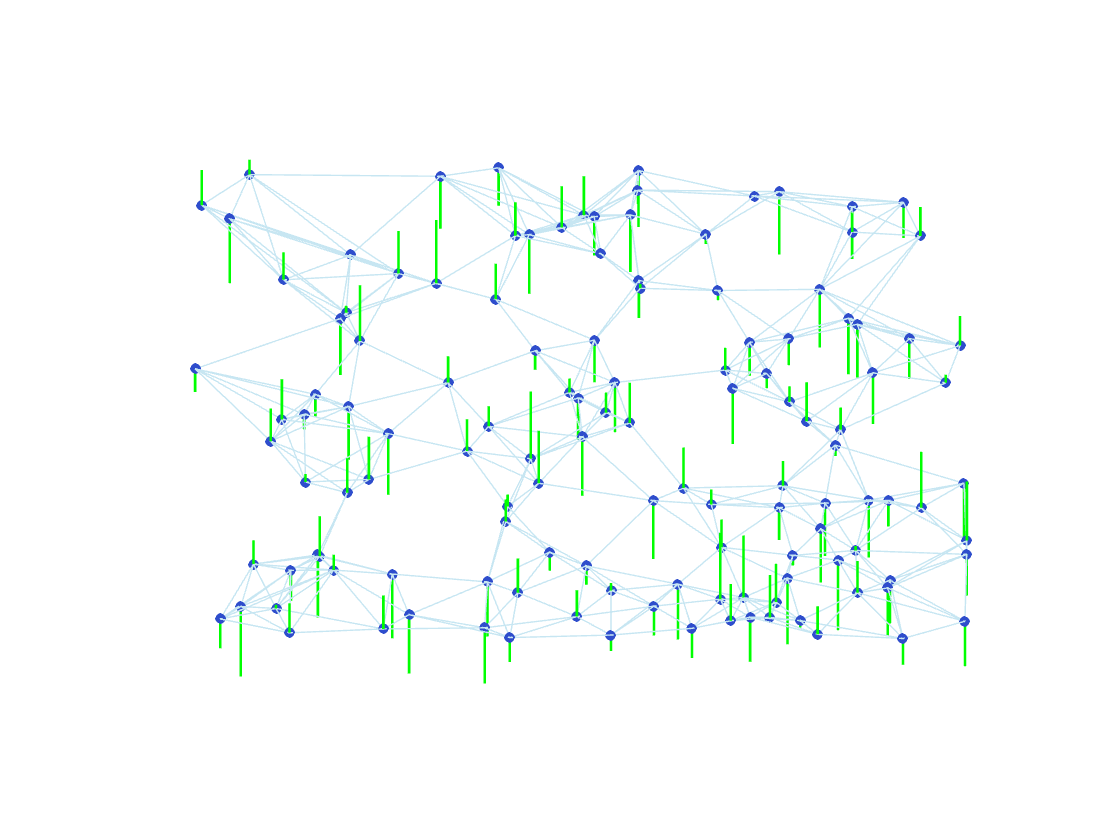}}
	\subfloat[]{
		\includegraphics[width=1.8 in]{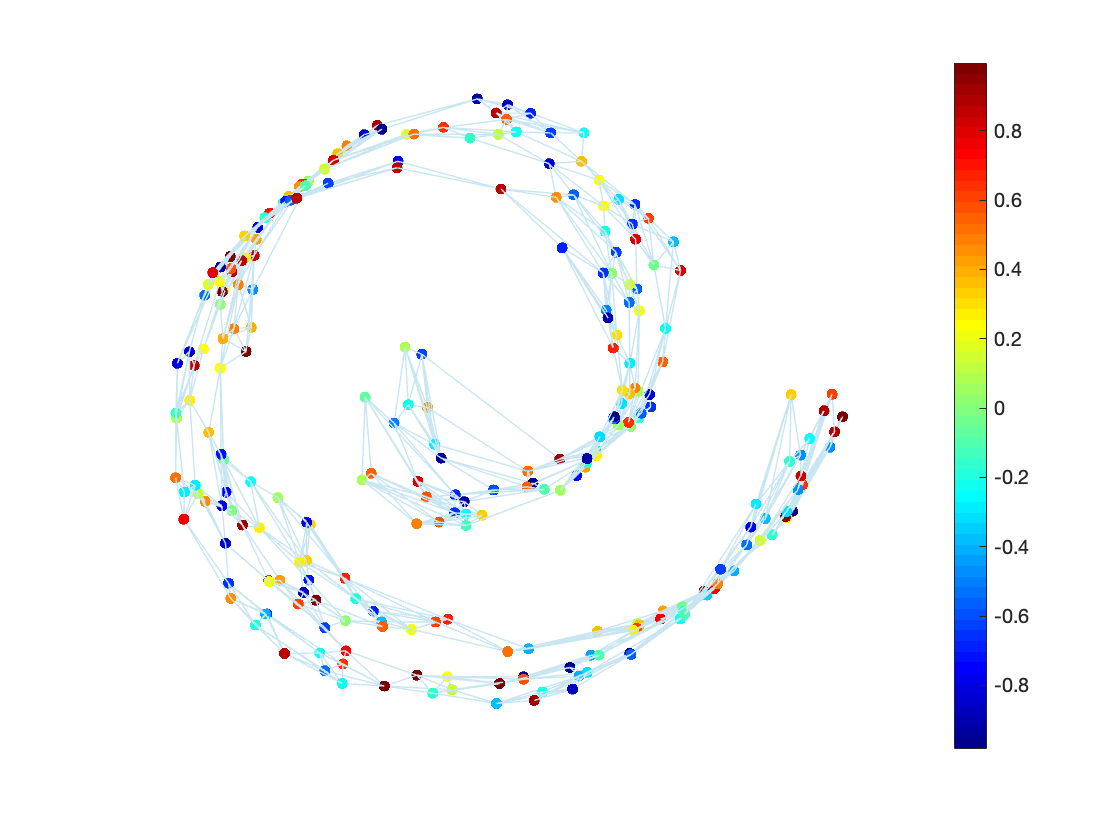}}
	\quad
	\subfloat[]{
		\includegraphics[width=1.8 in]{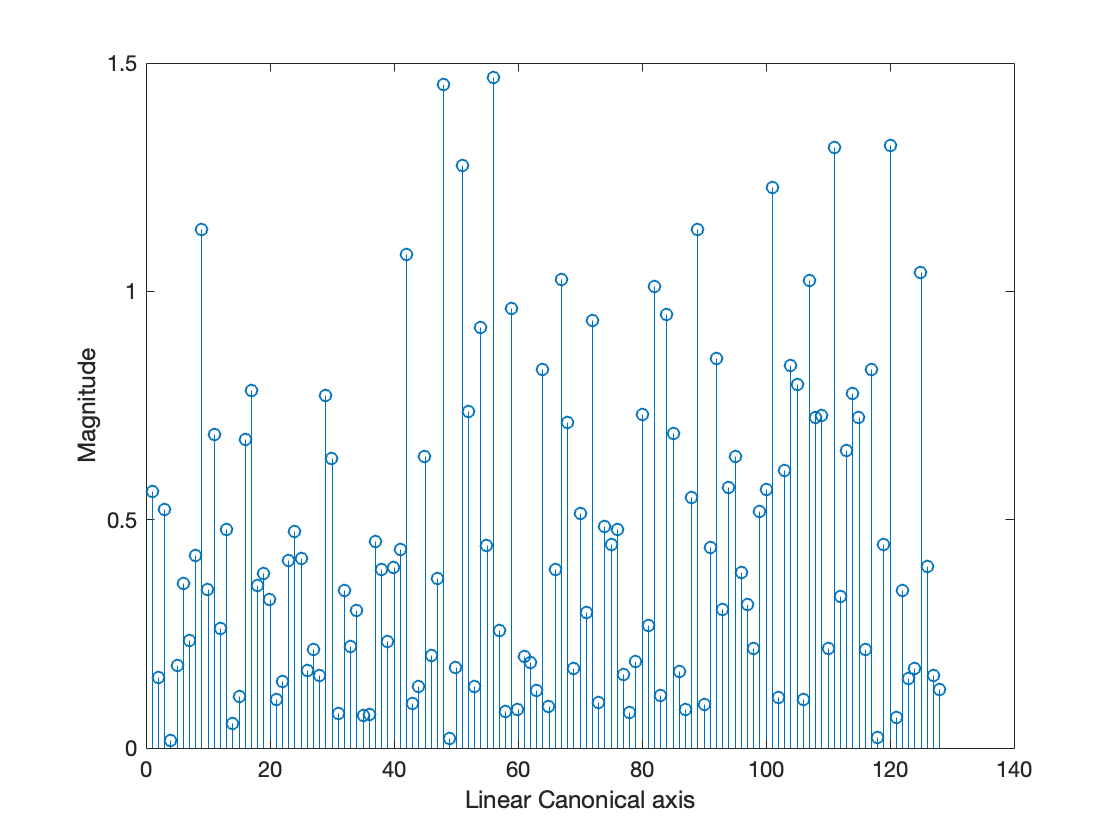}}
	\subfloat[]{
		\includegraphics[width=1.8 in]{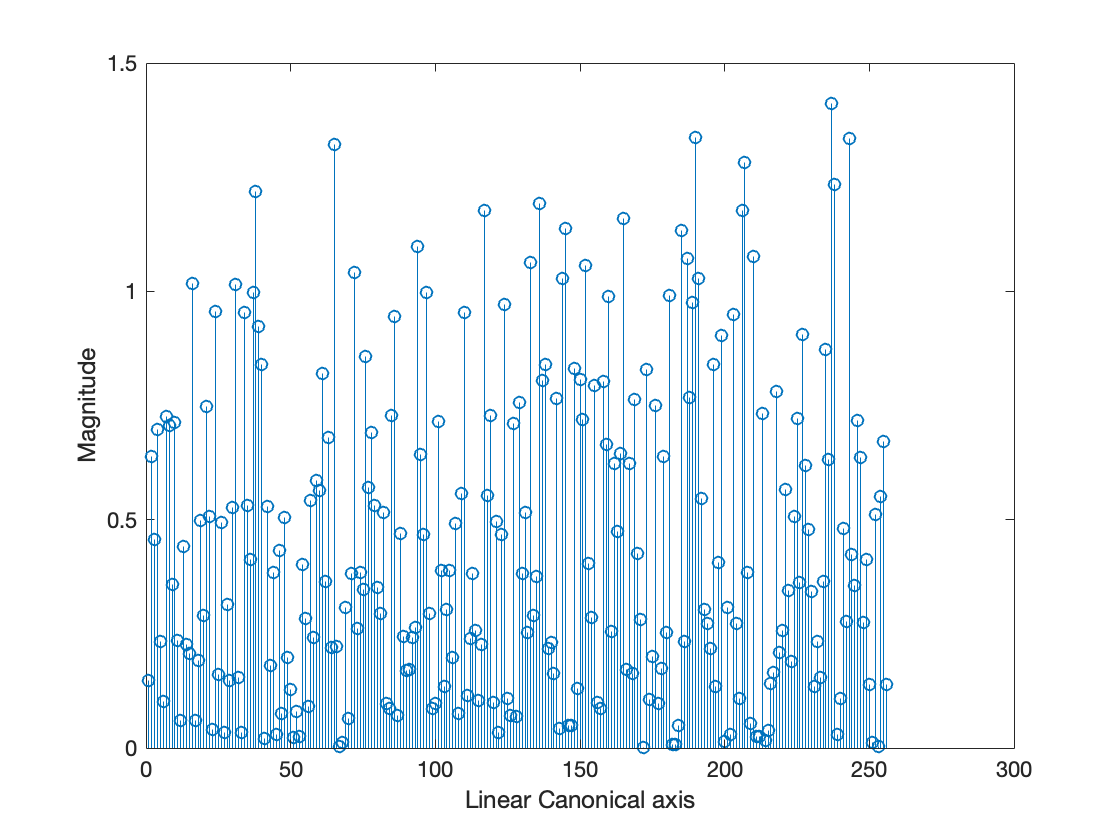}}
	\caption{(a) Sensor network graph. (b) Swiss roll graph. (c) A random signal on (a). (d) A random signal on (b). (e) The GLCT with $\sigma=1/3$, $\xi=0$, $\beta=0.8$ of (c). (f) The GLCT with $\sigma=1/3$, $\xi=0$, $\beta=0.8$ of (d).}\label{fig1}
\end{figure}

\begin{figure}
	\centering
	\subfloat[]{  
		\includegraphics[width=1 in]{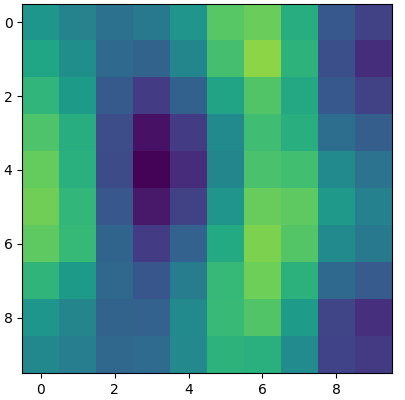}}
	\subfloat[]{
		\includegraphics[width=1 in]{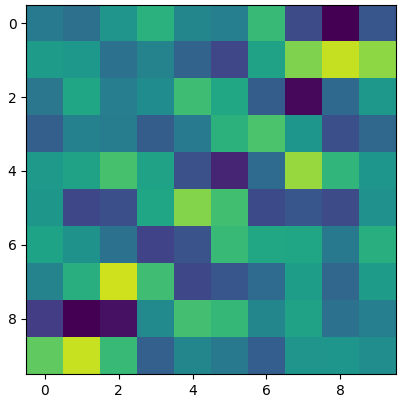}}
	\subfloat[]{
		\includegraphics[width=1 in]{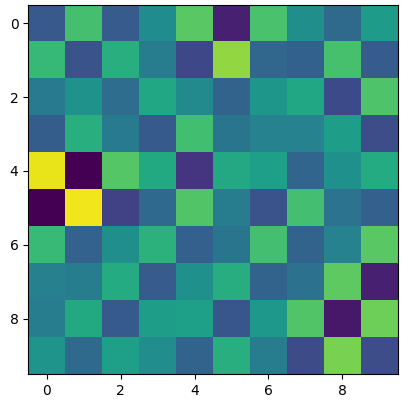}}
	\quad
	\subfloat[]{
		\includegraphics[width=1 in]{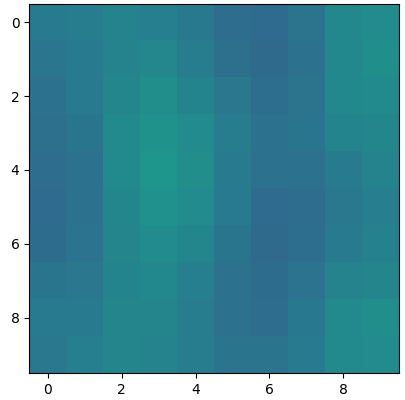}}
	\subfloat[]{
		\includegraphics[width=1 in]{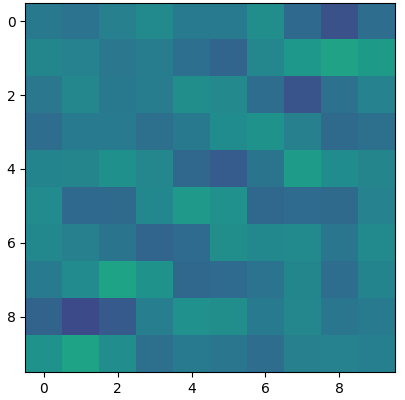}}
	\subfloat[]{
		\includegraphics[width=1 in]{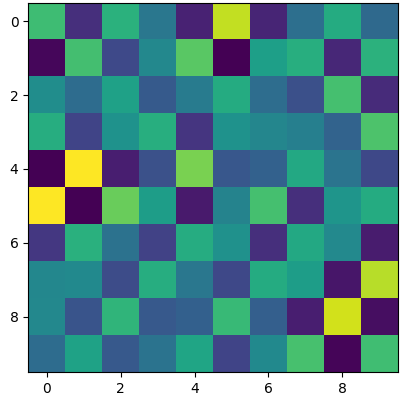}}
	\quad
	\subfloat[]{
		\includegraphics[width=1 in]{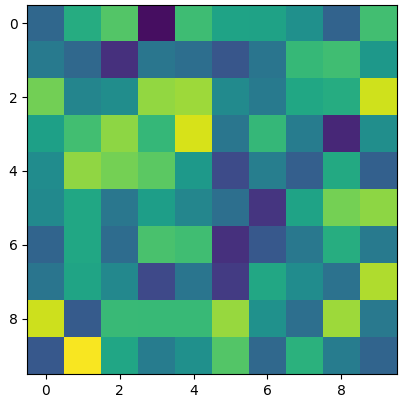}}
	\subfloat[]{
		\includegraphics[width=1 in]{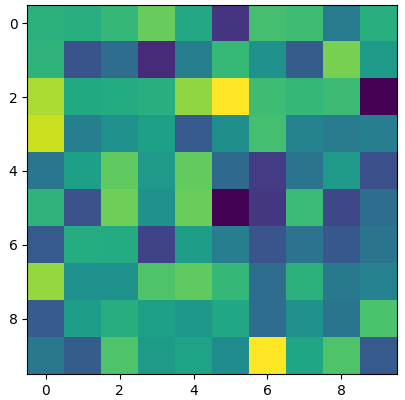}}
	\subfloat[]{
		\includegraphics[width=1 in]{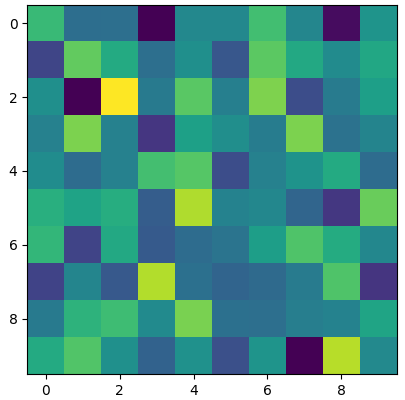}}
	\caption{Reshape the eigenvector to a 10×10 matrix. (a-c) The 9th, 49th, and 89th eigenvectors of the GFT. (d-f) The 9th, 49th, and 89th eigenvectors of GLCT with $\sigma=1$, $\xi=0.8$, $\beta=0.95$. (g-i) The 9th, 49th, and 89th eigenvectors of GLCT with $\sigma=0.4$, $\xi=0$, $\beta=0.95$. }\label{fig1_2}
\end{figure}

\section{Vertex-Frequency Analysis method associated with GLCT}
\label{sec4}
 Like classical time-frequency analysis, the spectral transformation of signals localized around a specific vertex $n$, forms the fundamental formula for vertex-frequency analysis. This spectral transformation is typically achieved using a localization window. 
Hence, a generalized method for vertex-frequency analysis, denoted as VGLCT, can be derived by combining the GLCT with local window functions \cite{stankovic2017vertex, stankovic2020vertex}. The VGLCT is obtained by computing the GLCT of a signal $f(n)$ and multiplying it by a suitable vertex localization window function, $h_i(n)$, resulting in the following expression
\begin{align}
	\mathcal{V}_h^Af(i,k)= \sum_{n=0}^{N-1} \sum_{l=0}^{N-1} f(n) h_i(n) \epsilon_k p_l(k) \kappa_l q_l^*(n).
\end{align}

In general, the VGLCT initially localizes the signal around vertex $i$ using the graph window function and then evaluates its GLCT performance. Specifically, the graph window function should have a value close to 1 at vertex $i$, indicating that the vertex is in close proximity to vertex $i$ or its nearest neighbors. Conversely, for vertices that are distant from vertex $i$, the graph window function should have a value close to 0, indicating negligible influence on the analysis.

We can categorize the methods for constructing the graph window into three categories:
\begin{itemize}
\item The window function is constructed based on the convolution and translation operators of GLCT, leading to the definition of windowed GLCT.
\item By utilizing the spectral shift form of the kernel of the bandpass function in the linear regular domain of the graph, we can define GLCT-based wavelet transform and S-transform.
\item A window function constructed based on the vertex neighborhood can define local GLCT.
\end{itemize}

\subsection{Windowed graph linear canonical transform} \label{sec3.1}
Unlike the traditional Fourier transform, the windowed Fourier transform (or short-time Fourier transform) divides the signal into a series of short-time windows and applies the Fourier transform to the signal within each window to obtain the spectral information of each window. In graph signal processing, the windowed graph Fourier transform (WGFT) has also been proposed \cite{shuman2012windowed, shuman2016vertex}. The basic idea is to define the WGFT as the inner product of the signal $f$ and a shifted and modulated window function $h$. The WGFT is defined by incorporating convolution, translation, and modulation operations into the graph signal \cite{yan2021windowed}. In this section, we begin by defining the graph linear canonical convolution operator and the graph linear canonical translation operator. Subsequently, we define the windowed graph linear canonical convolution transform (WGLCT) by multiplying the signal f with the translated window function prior to applying the GLCT.

\textbf{Graph linear canonical convolution operator.}
Consider two signals, $f(n)$ and $h(n)$, defined on a graph, with their corresponding GLCTs represented as $\hat{f}_A$ and $\hat{h}_A$. The graph linear canonical convolution $*_A$ of these signals is defined as follows
\begin{align} \label{eq9}
	&(f *_A h)(n) = IGLCT\{ \hat{f}_{A}(k) \hat{h}_{A}(k) \} \nonumber \\
	&= \sum_{k=0}^{N-1} \sum_{l=0}^{N-1} \hat{f}_{A}(k) \hat{h}_{A}(k) q_l(n) \kappa_l^{*} p_l^*(k) \epsilon_k^{*}.
\end{align}

The concept of 'shift' on a graph cannot be directly and uniquely extended to graph signals using an analogy with classical signal shifts. In the realm of vertex frequency analysis, there has been an intriguing effort to generalize convolution and define corresponding shift operators on graphs. In this context, we aim to further generalize the calculation of shifts based on graph linear regularization transforms.

\textbf{Graph linear canonical translation operator.}
For any signal $h(n)$ defined on a graph, and the delta function located at a vertex $ i \in \{0,1,\cdots,N-1\}$, given by $\delta_i(n)=\delta(n-i)$. The GLCT of delta function $\delta_i(n)$, is given by
\begin{align}
	\Delta_i(k)
	&=GLCT\{\delta_i(n)\}
	= \sum_{n=0}^{N-1} \sum_{l=0}^{N-1} \delta_i(n)  \epsilon_k p_l(k) \kappa_l q_l^*(n) \nonumber \\
	&= \sum_{l=0}^{N-1}\epsilon_k p_l(k) \kappa_l q_l^*(i).
\end{align}
The graph linear canonical translation $T_i^A$ is defined as
\begin{align}
	\label{eq2}
	&T_i^A h(n)=(h *_A \delta_i)(n) 
	= \sum_{k=0}^{N-1} \sum_{l=0}^{N-1} \hat{h}_{A}(k) \nonumber \\
	& \times
	\left(\sum_{s=0}^{N-1}\epsilon_k p_s(k) \kappa_s q_s^*(i)\right)
	 q_l(n) \kappa_l^{*} p_l^*(k) \epsilon_k^{*}.
\end{align}

\begin{figure}
	\centering
	\subfloat[]{  
		\includegraphics[width=1.6 in]{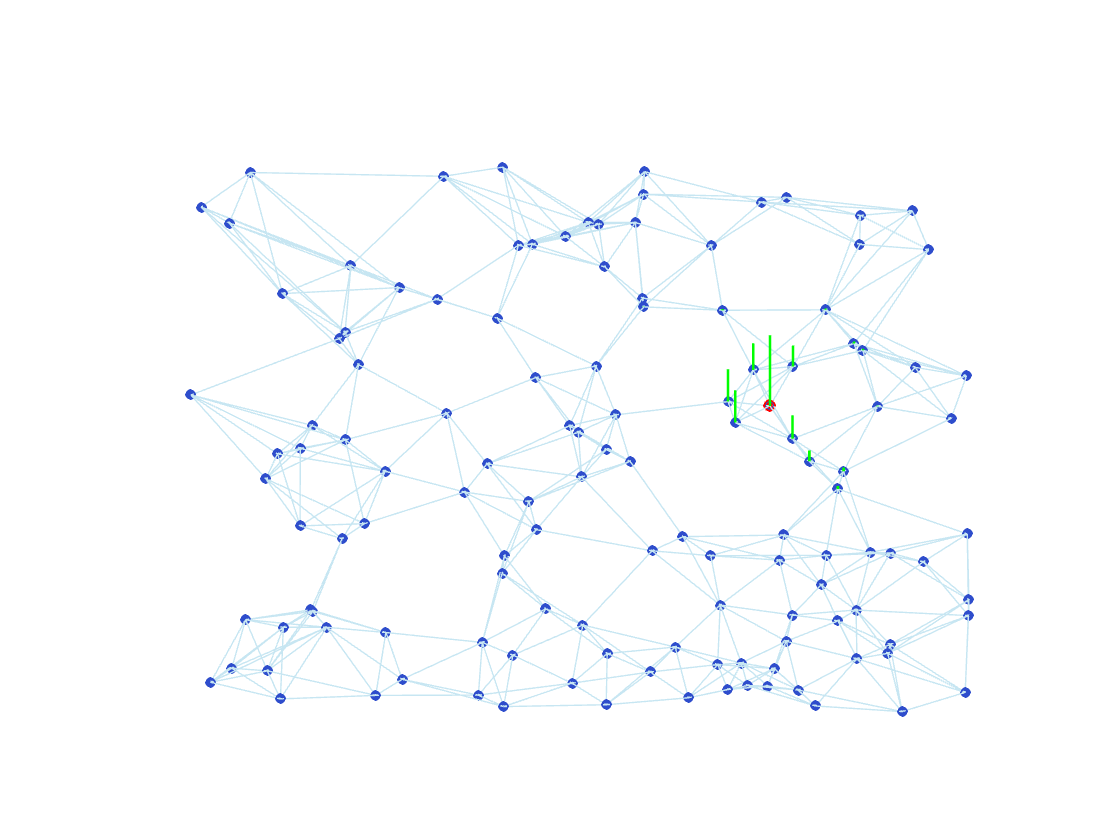}}
	\subfloat[]{
		\includegraphics[width=1.6 in]{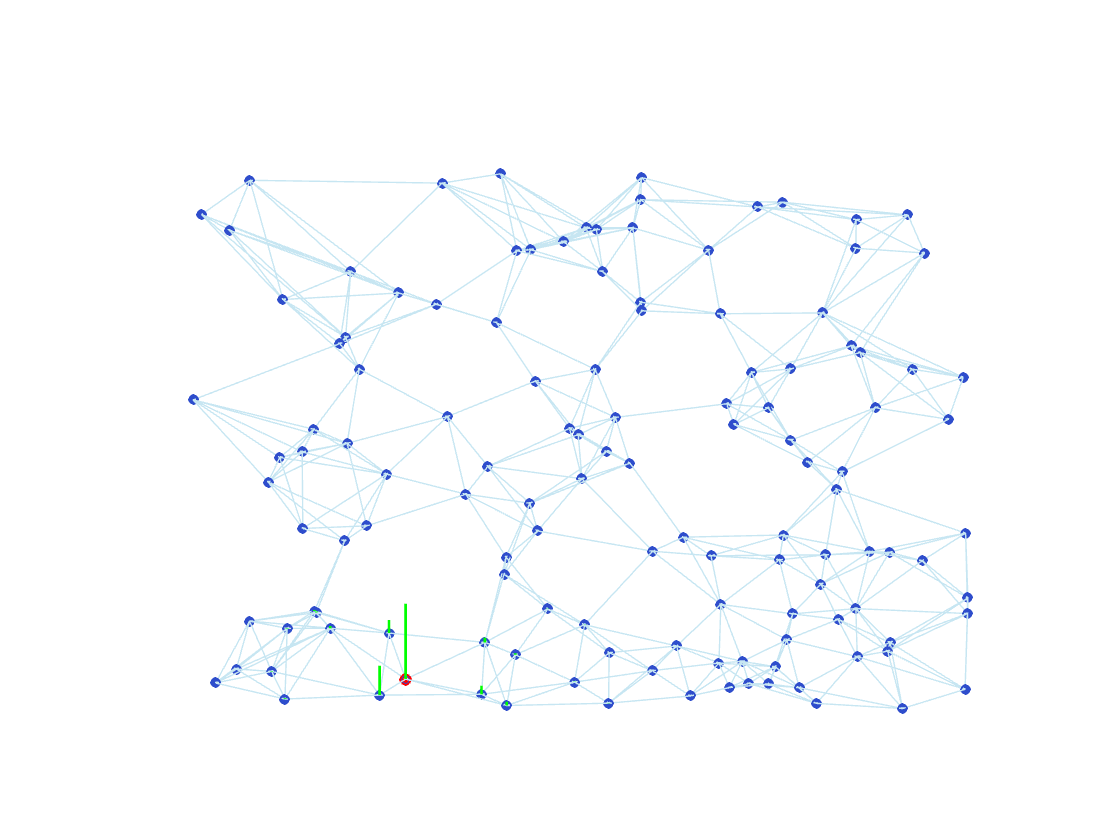}}
	\quad
	\subfloat[]{  
		\includegraphics[width=1.6 in]{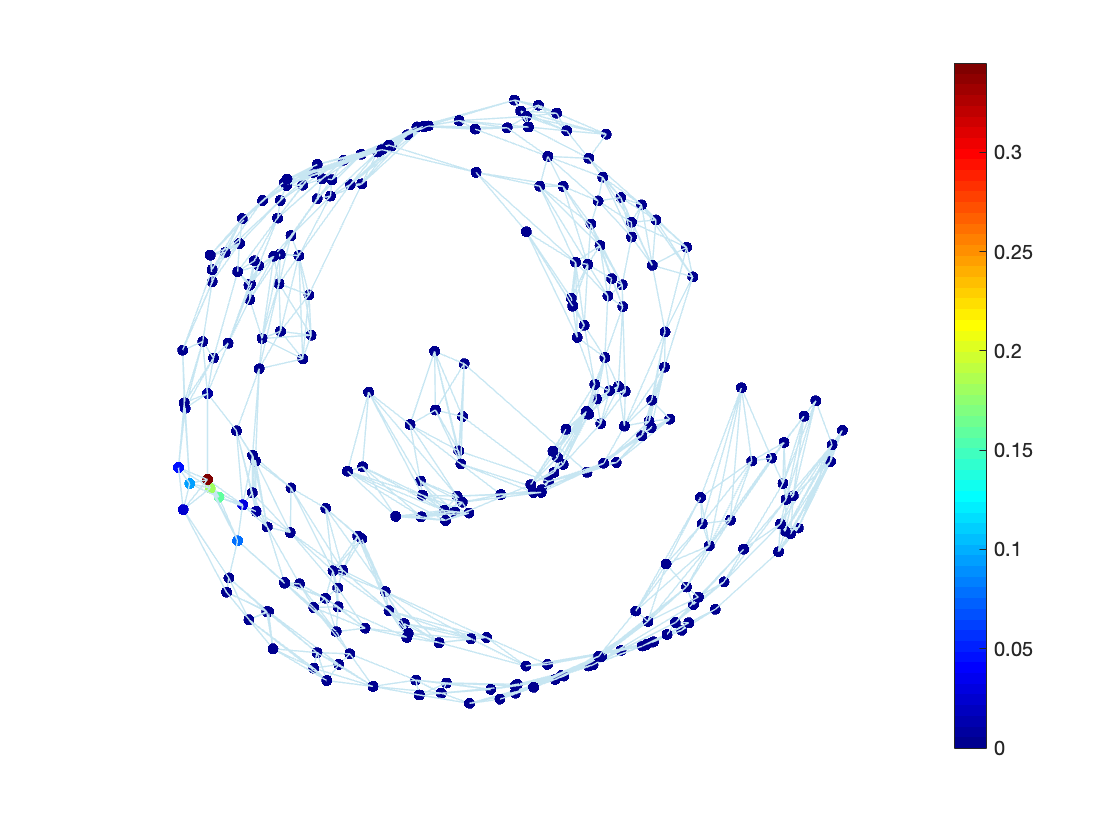}}
	\subfloat[]{
		\includegraphics[width=1.6 in]{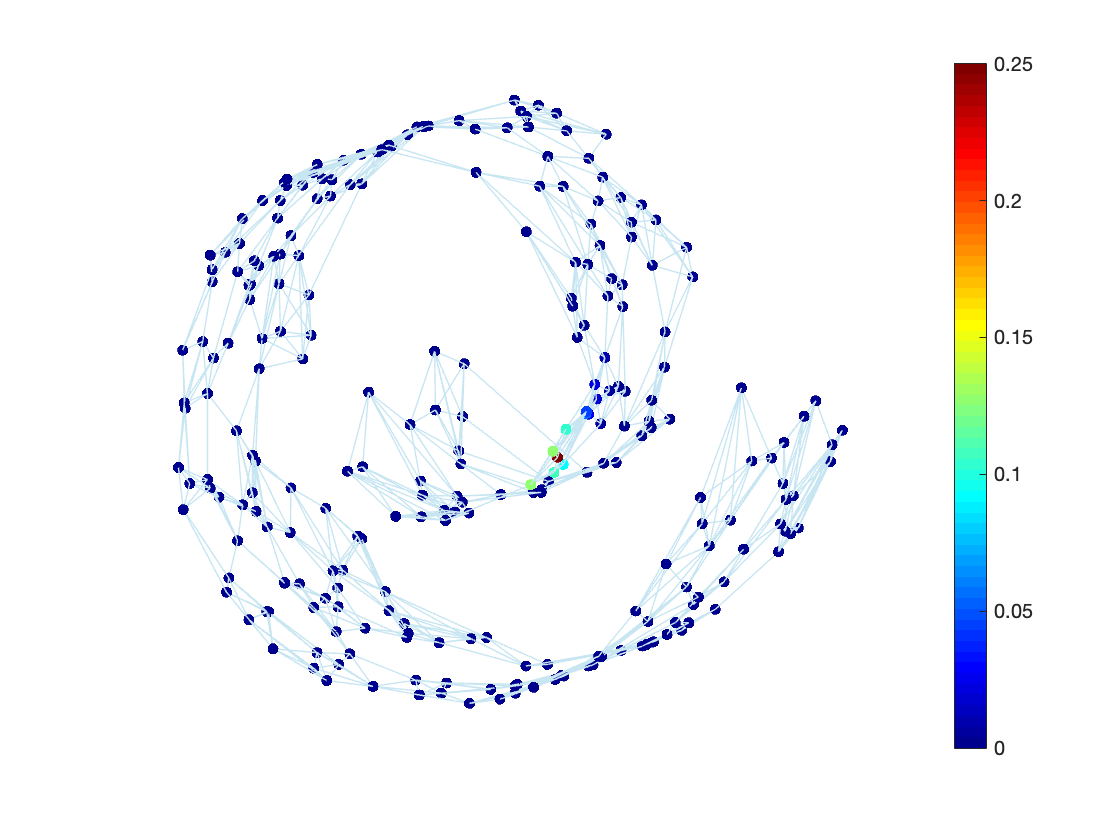}}
	\caption{Graph linear canonical translation. The red dot in (a) (b) represents the selected vertex $i$, while the green line represents the magnitude of the function value. (a) $T_{30}^A h$ on sensor network graph. (b) $T_{90}^A h$ on sensor network graph. (c) $T_{30}^A h$ on swiss roll graph. (d) $T_{200}^A h$ on swiss roll graph. }\label{fig2}
\end{figure}

Therefore, the WGLCT can be written as 
\begin{align}
	\mathcal{V}_{h,1}^Af(i,k)
	= \sum_{n=0}^{N-1} \sum_{l=0}^{N-1} f(n) T_i^A h^*(n) \epsilon_k p_l(k) \kappa_l q_l^*(n)
\end{align}
where $T_i^A h(n)$ is defined in (\ref{eq2}).
By applying the graph window function $T_i^A h(n)$ to the signal $f(n)$ on each vertex, and then summing over all vertices $i$, a constant overlap-add relation is achieved. Then the inverse WGLCT is
\begin{align}
	f(n) 
	= \frac{1}{\sum_{i=0}^{N-1}T_i^A h(n)}\sum_{k=0}^{N-1} \sum_{l=0}^{N-1} \mathcal{V}_{T_i^A h}^Af(i,k)
	q_l(n) \kappa_l^{*} p_l^*(k) \epsilon_k^{*}.
\end{align}

The basic form of this window, $h(n)$, can be conveniently defined in the spectral domain, for example as
$H(k) = c e^{-\lambda_k \tau}$,
where $c$ denotes the window amplitude and $\tau > 0$ is a constant which determines the window width.

\begin{example} \label{exam4}
	 Let $H(k) = e^{-\lambda_k/2}$ be the spectral domain-defined window, and $\sigma=1/3$, $\xi=0$, $\beta=0.8$. For the sensor network graph depicted in Fig. \ref{fig1} (a) in Example \ref{exam1}, we performed local point frequency analysis on vertex 30 and vertex 90, respectively. The results are displayed in Fig. \ref{fig2} (a) and (b). For the Swiss roll graph shown in Fig. \ref{fig1} (b), we conducted local point frequency analysis on vertex 30 and vertex 200, respectively. The corresponding results are presented in Fig. \ref{fig2} (c) and (d).
\end{example}

\begin{example} \label{exam5}
	Consider the real part of graph signal as $f_1(n)={\Re}(\varsigma_{30}(n)) + {\Re}(\varsigma_{80}(n))$
	for the sensor network graph on Fig. \ref{fig1} (a), where $\varsigma_{k}(n)=\sum_{l=0}^{N-1} \epsilon_{k} p_l(k) \kappa_l q_l^*(n)$ is derived from the graph linear canonical Laplacian matrix of the sensor network graph. Let $H(k) = e^{-400\lambda_k}$, $\sigma=1/3$, $\xi=0$, and $\beta=0.8$, the WGLCT of $f_1$ is depicted in Figure \ref{fig3} (a). 
\end{example}
\begin{example}
	Consider a signal on the Swiss roll graph shown in Fig. \ref{fig1} (b), $f_2(n)={\Re}(\varsigma_{20}(n)) + {\Re}(\varsigma_{50}(n))+{\Re}(\varsigma_{70}(n))$, where $\varsigma_{k}(n)=\sum_{l=0}^{N-1} \epsilon_{k} p_l(k) \kappa_l q_l^*(n)$. Here, $\varsigma_{k}(n)$ is derived from the graph linear canonical Laplacian matrix of the Swiss roll graph.
	Let $H(k) = e^{-3500\lambda_k}$, $\sigma=1/3$, $\xi=0$, and $\beta=0.8$, the WGLCT of $f_2$ is depicted in Figure \ref{fig3} (b). 
\end{example}

Below, we present a rapid computation approach for WGLCT.
In the traditional scenario of one-dimensional signals, the windowed Fourier transform can be computed by taking the inverse FT from the so-called $\alpha$-domain \cite{yan2021windowed, brown2009general, jestrovic2017fast}. The $\alpha$-domain is defined as the FT of the windowed Fourier domain. 
We define $\alpha$-domain in linear canonical domain by
\begin{align}
	\label{eq5}
	\alpha(u,u')=  \int_{-\infty}^{\infty} V_h^{A} f(l, u) K_{A}(l, u') {\rm d} l
\end{align}
where $V_h^{A} f(l, u)$ represents the windowed LCT of $f$ \cite{kou2012windowed}, that is, (\ref{eq5}) applies LCT to $V_h^{A} f(l, u)$. The fast WLCT can be obtained by applying inverse LCT to $\alpha$ on the variable $u'$
\begin{align}
	FV_h^{A} f(l, u) = \int_{-\infty}^{\infty} \alpha(u,u') K_{A^{-1}}(l, u') {\rm d} u'.
\end{align}

Similarly, we define the $\alpha$-domain as the GLCT of WGLCT in graph linear canonical domain as follows
\begin{align}
	\alpha(k,k')
	= \sum_{n=0}^{N-1} \sum_{l=0}^{N-1}\mathcal{V}_{T_i^A h}^Af(n,k')  \epsilon_k p_l(k) \kappa_l q_l^*(n).
\end{align}
Then, the discretization of $\alpha$ is 
\begin{align}
	\alpha(k,k')
	=& \sum_{n=0}^{N-1} \sum_{l=0}^{N-1}
	\sum_{m=0}^{N-1} \sum_{r=0}^{N-1} f(m)
	\sum_{d=0}^{N-1} \sum_{y=0}^{N-1} \sum_{s=0}^{N-1} \hat{h}_{A}^*(d) \nonumber \\
	& \times
	q_s(n) \kappa_s^{*} p_s^*(d) \epsilon_{d}^{*}
	\epsilon_{d} p_y(d) \kappa_y q_y^*(m)
	\nonumber \\
	& \times 
	\epsilon_{k'} p_r(k') \kappa_r q_r^*(m) 
	\epsilon_k p_l(k) \kappa_l q_l^*(n) \nonumber \\
	= & \sum_{m=0}^{N-1} \sum_{r=0}^{N-1} 
	\sum_{d=0}^{N-1} \sum_{y=0}^{N-1} \sum_{s=0}^{N-1}
	f(m)  \hat{h}_{A}(d)
	\epsilon_{k'} p_r(k') \kappa_r q_r^*(m) \nonumber \\
	&\times
	\epsilon_{d} p_y(d) \kappa_y q_y^*(m)
	\delta_{dk}
	\nonumber \\
	= & \sum_{m=0}^{N-1}  \sum_{r=0}^{N-1} \sum_{y=0}^{N-1} 
	f(m)  \hat{h}_{A}(k)
	\epsilon_{k'} p_r(k') \kappa_r q_r^*(m) \nonumber \\
	&\times
	\epsilon_{k} p_y(k) \kappa_y q_y^*(m).
\end{align}
The fast WGLCT is achieved by performing the inverse GLCT on the variable $k$ with respect to $\alpha$, that is
\begin{align}
	&F\mathcal{V}_{h,1}^Af(i,k') 
	= \sum_{k=0}^{N-1} \sum_{l=0}^{N-1} \alpha(k,k')
	q_l(n) \kappa_l^{*} p_l^*(k) \epsilon_k^{*} \nonumber \\
	&= \sum_{k=0}^{N-1} \sum_{l=0}^{N-1} 
	\tilde{f}(k,k')  \hat{h}_{A}(k)
	q_l(n) \kappa_l^{*} p_l^*(k) \epsilon_k^{*}
\end{align}
where 
\begin{align}
	\tilde{f}(k,k') 
	&=\sum_{m=0}^{N-1}  \sum_{r=0}^{N-1} \sum_{y=0}^{N-1} f(m) \epsilon_{k'} p_r(k') \kappa_r q_r^*(m) \nonumber \\
	& \times \epsilon_{k} p_y(k) \kappa_y q_y^*(m).
\end{align}

\begin{algorithm}[!h]
	\caption{Fast windowed graph linear canonical transform}
	\begin{algorithmic}[1]
		\REQUIRE ~~\\
		Graph signal $f$. \\
		An $N\times N$ matrix $\mathbf{F}$ (each row of $\mathbf{F}$ is the signal $f$).\\
		\ENSURE ~~\\
		The fast WGLCT matrix $\mathbf{FV}_{h,2}^A$.
		\STATE Compute $\mathcal{F}_{A} =  \Upsilon \mathbf{P} \mathcal{K} \mathbf{Q}^H$.
		\STATE Calculate $\tilde{f}(k,k')$ using the expression $\tilde{\mathbf{F}} = \lceil \tilde{f}(k,k') \rceil = (\mathbf{F} \circ \mathcal{F}_{A}^H) \cdot \mathcal{F}_{A}$, where $\circ$ represents element-wise multiplication (Hadamard multiplication), and $\cdot$ denotes standard matrix multiplication.
		\STATE Construct a matrix where each row corresponds to the GLCT of the window function \\
		$\mathbf{H}_A=\left(\begin{array}{cccc}\hat{h}_A\left(r_0\right) & \hat{h}_A\left(k_1\right) & \ldots & \hat{h}_A\left(k_{N-1}\right) \\ \hat{h}_A\left(k_0\right) & \hat{h}_A\left(k_1\right) & \ldots & \hat{h}_A\left(k_{N-1}\right) \\ \vdots & \vdots & \ddots & \vdots \\ \hat{h}_A\left(k_0\right) & \hat{h}_A\left(k_1\right) & \ldots & \hat{h}_A\left(k_{N-1}\right)\end{array}\right)$.
		\STATE Compute the $\alpha$-domain representation by multiplying $\tilde{\mathbf{F}}$ and $\mathbf{H}_A^*$, i.e., $\alpha = \mathbf{H}_A^* \cdot \tilde{\mathbf{F}}$.
		\STATE The fast WGLCT at each vertex is obtained by applying inverse GLCT to $\sigma$, i.e., $\mathbf{FV}_{h,2}^A = \mathcal{F}_{A} \cdot \alpha$.
	\end{algorithmic}
\end{algorithm}

\begin{example}
	To compare the results of directly calculating the WGLCT with those obtained through the fast calculation method, we utilize the signals $f_1$ and $f_2$ from Example \ref{exam4} and Example \ref{exam5}, respectively. As depicted in Figure \ref{fig9}, it is evident that both the fast calculation and direct calculation methods yield the same results, indicating that they have an equivalent effect.
\end{example}

\begin{figure}
	\centering
	\subfloat[]{  
		\includegraphics[width=1.6 in]{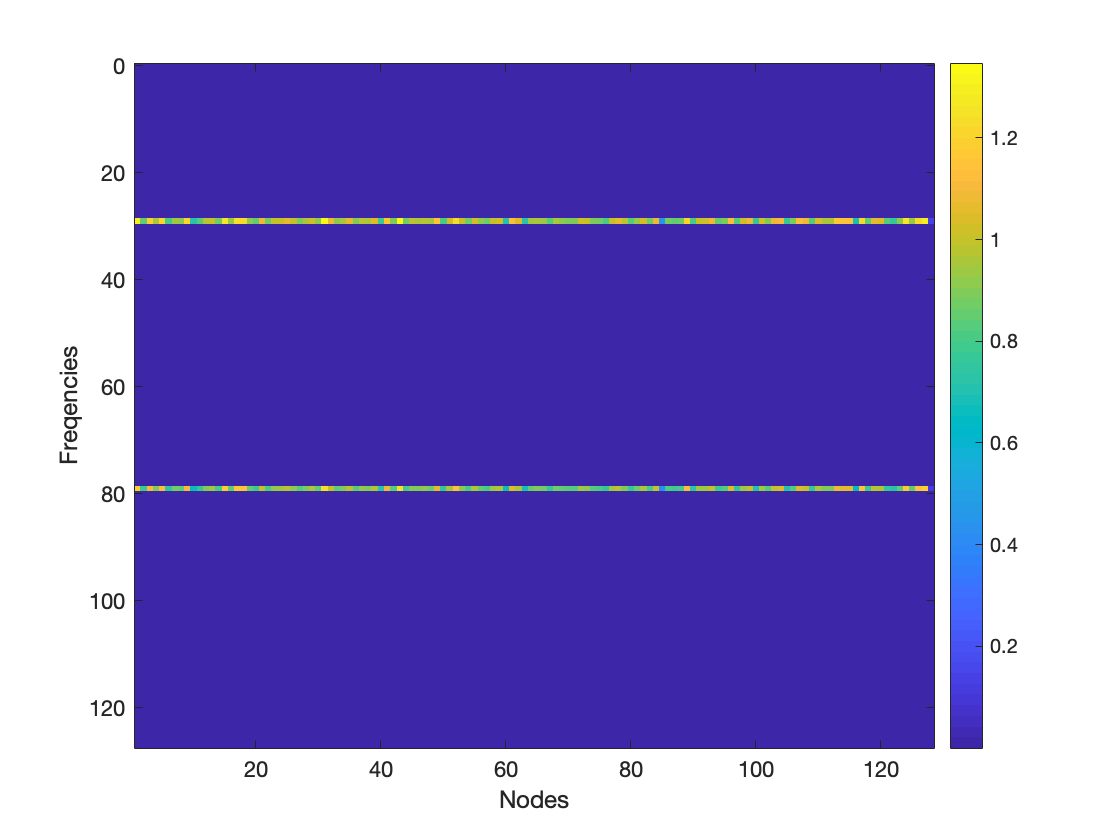}}
	\subfloat[]{
		\includegraphics[width=1.6 in]{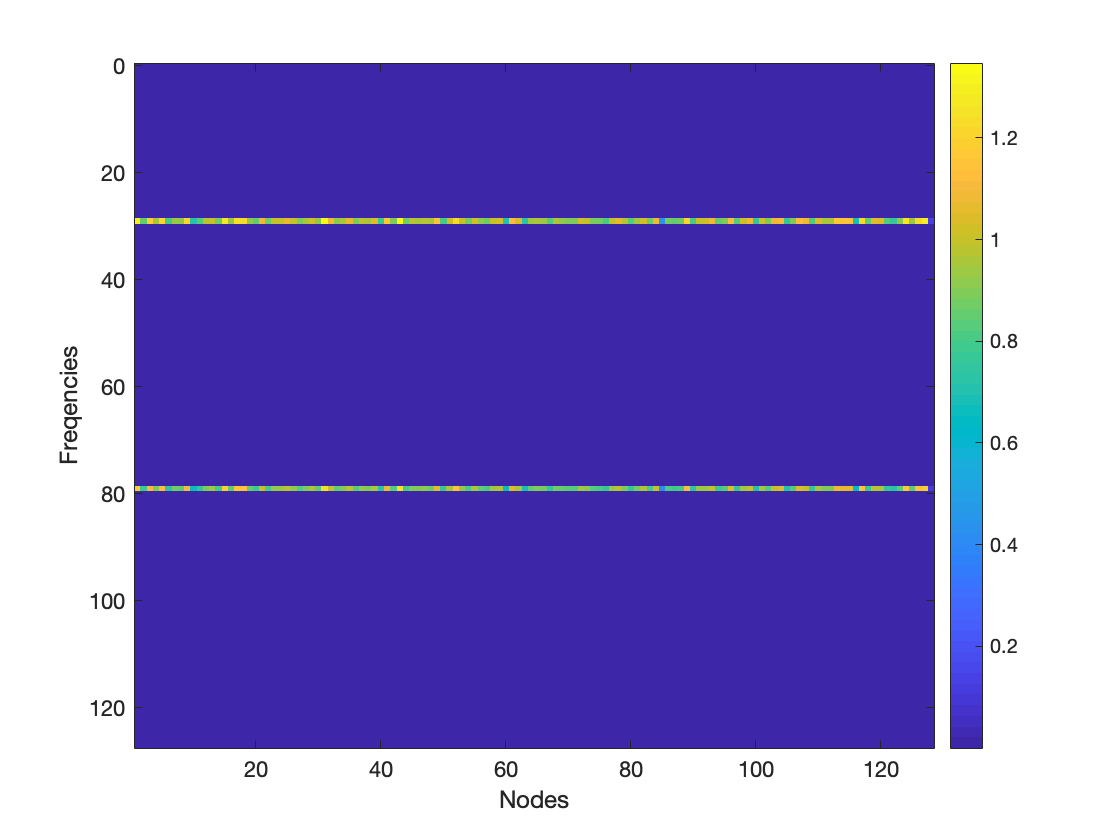}}
	\quad
	\subfloat[]{  
		\includegraphics[width=1.6 in]{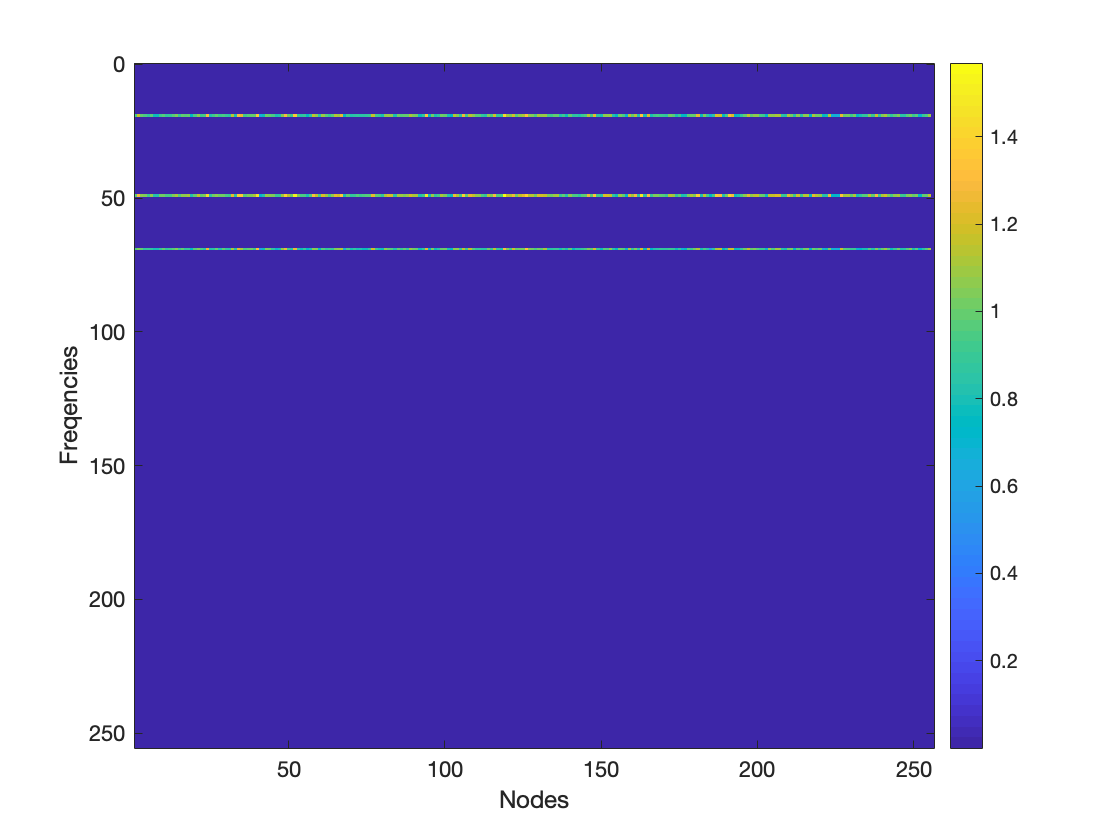}}
	\subfloat[]{
		\includegraphics[width=1.6 in]{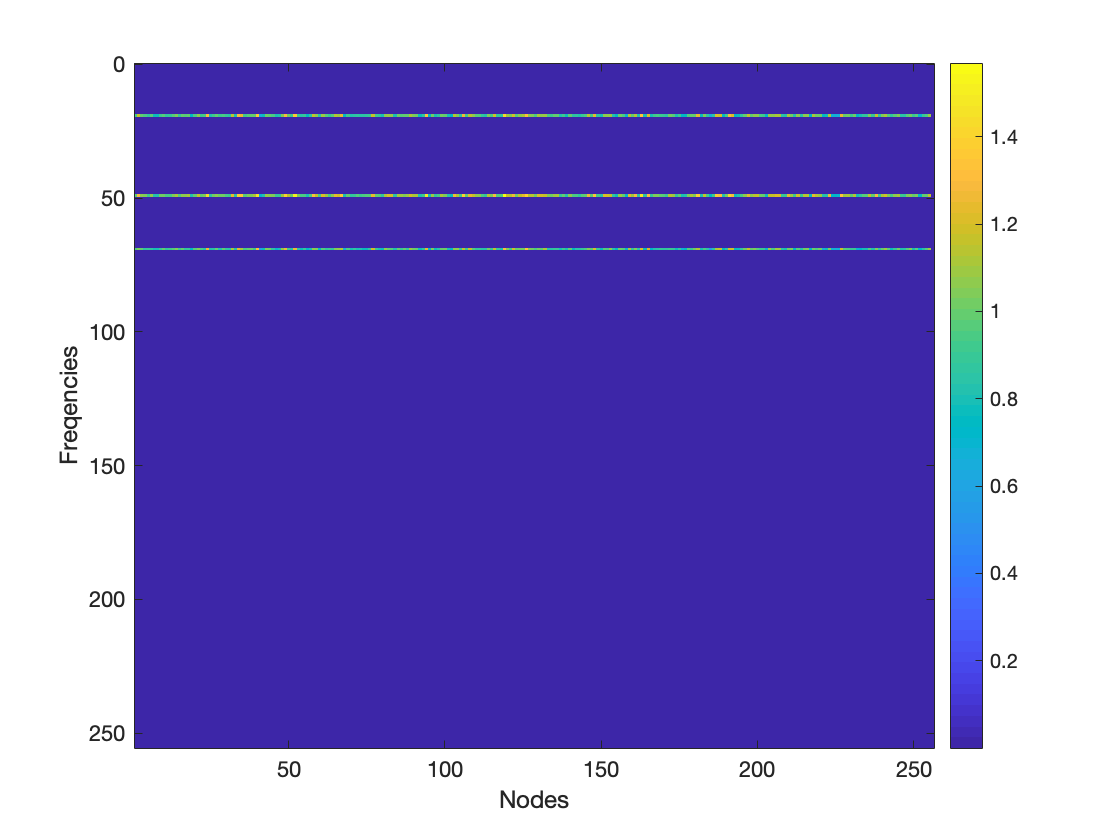}}
	\caption{(a)The WGLCT of $f_1$ in Example \ref{exam4}. (b)The fast WGLCT of $f_1$ in Example \ref{exam4}. (c)The WGLCT of $f_2$ in Example \ref{exam5}. (d)The fast WGLCT of $f_2$ in Example \ref{exam5}. }\label{fig9}
\end{figure}

\begin{figure}
	\centering
	\subfloat[]{  
		\includegraphics[width=1.6 in]{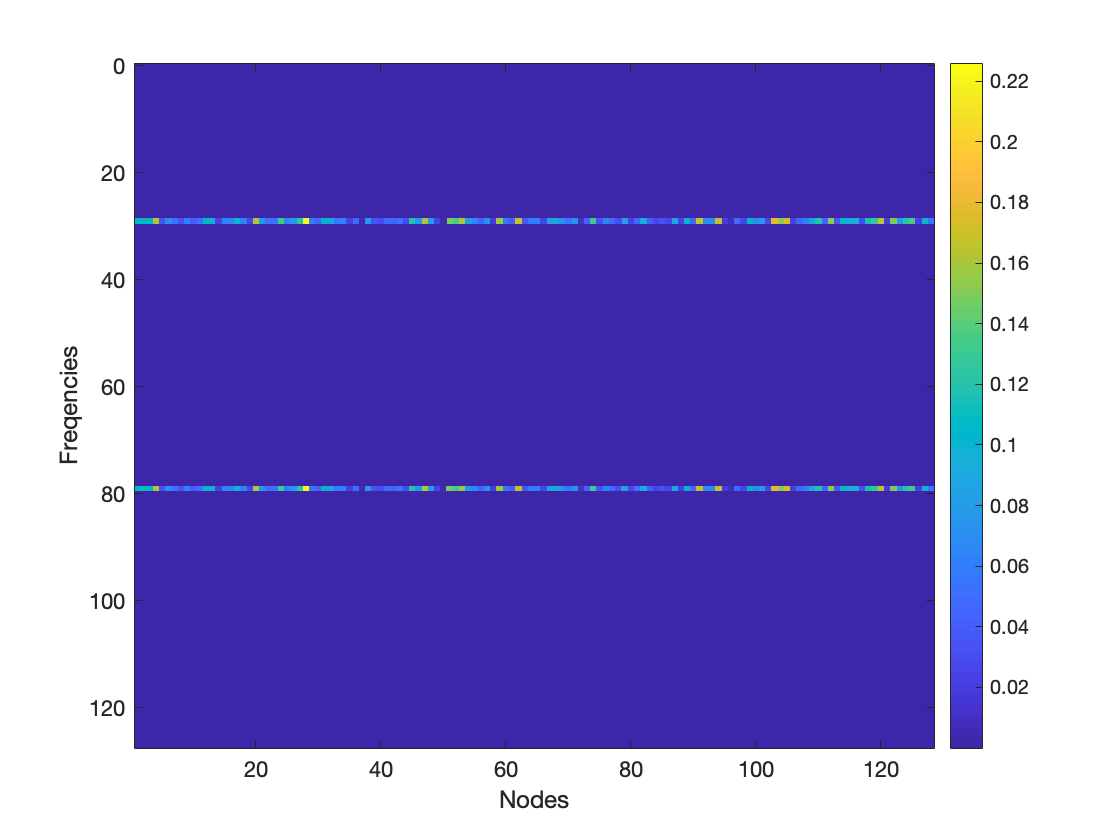}}
	\subfloat[]{
		\includegraphics[width=1.6 in]{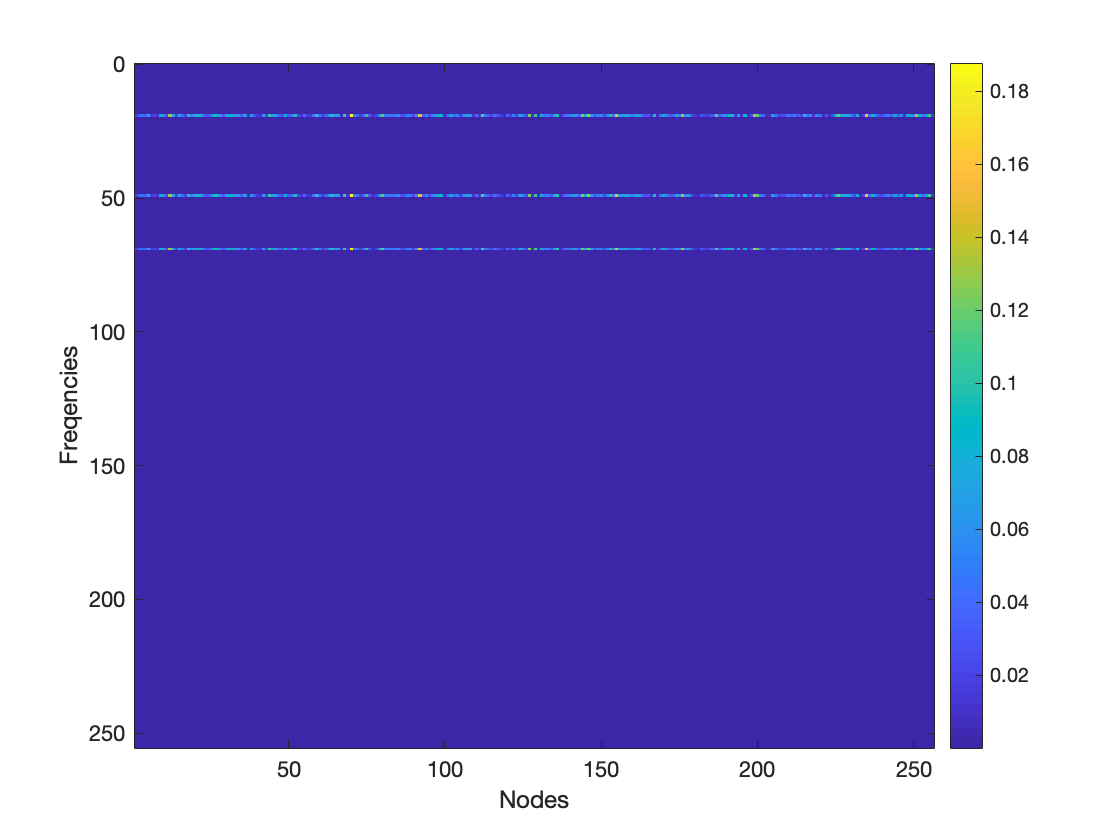}}
	\caption{The WGLCT of graph signals. (a) The WGLCT of a signal on sensor network graph. (b) The WGLCT of a signal on swiss roll graph. }\label{fig3}
\end{figure}

\subsection{Graph linear canonical domain localization} 
\label{sec3.2}
The classical time-frequency analysis method employs windows for frequency localization in the spectral domain, enabling the derivation of the dual form of the time-frequency analysis method using the DFT of the original signal and the spectral window. Similarly, for graph signals, we can apply this method to localize them in the graph linear canonical domain and obtain VGLCT. 
For graph signals, we can also employ this approach to achieve localization in the spectral domain. In this case, the VGLCT is obtained as the IGLCT of $\hat{f}(z)$, which is localized using a spectral domain window $H(k - z)$ centered around the spectral index $k$ \cite{stankovic2020vertex}. Thus, the VGLCT can be expressed as follows
\begin{align}
	\mathcal{V}_h^Af(i,k)
	= \sum_{p=0}^{N-1} \sum_{l=0}^{N-1} \hat{f}_{A}(p) H(k-z)
	q_l(n) \kappa_l^{*} p_l^*(z) \epsilon_z^{*}.
\end{align}
It can be imlemented using band-pass transfer functions, $H(k - z)=H_k(\lambda_z)$, as
\begin{align}
	\mathcal{V}_{h,2}^Af(i,k)
	= \sum_{p=0}^{N-1} \sum_{l=0}^{N-1} \hat{f}_{A}(z) H_k(\lambda_z)
	q_l(n) \kappa_l^{*} p_l^*(z) \epsilon_z^{*}.
\end{align}

The spectral domain localization window for a band-pass graph system can be expressed as a transfer function using eigenvalues $\lambda_z$ as the center and its surroundings \cite{stankovic2019vertex, stankovic2020vertex}. It is defined in a polynomial form, represented by the following equation
\begin{align}
	\label{eq6}
	H_k(\lambda_z)=h_{0,k}+h_{1,k}\lambda_z+\cdots+h_{C-1,k}\lambda_z^{C-1}
\end{align}
where $C$ reprents the polynomial order, and $k=0,1,\cdots,K-1$ with K being the number of bands.
Then the elements $\mathcal{V}_{h,2}^Af(i,k), i=0,1,\cdots,N-1$ of the VGLCT matrix $\mathbf{V}_{h,2}^A$ can be written in a matrix form
\begin{align}
	\mathbf{v}_{h,2}^{k,A} 
	&= IGLCT\{ \hat{f}_{A}(z) H_k(\lambda_z)\}
	=  \mathbf{Q} \mathcal{K}^{-1} \mathbf{P}^H \Upsilon^{-1}\hat{\mathbf{f}}_A  H_k(\boldsymbol{\Lambda}_A) \nonumber \\
	&=  \mathbf{Q} \mathcal{K}^{-1} \mathbf{P}^H \Upsilon^{-1}
		H_k(\boldsymbol{\Lambda}_A) 
		\Upsilon \mathbf{P} \mathcal{K} \mathbf{Q}^H \mathbf{f} \nonumber \\
	&=H_k(\mathcal{L}_{A}) \mathbf{f}
	= \sum_{m=0}^{C-1} h_{m,k}\mathcal{L}_{A}^m \mathbf{f}
\end{align}
where $\mathbf{v}_{h,2}^{k,A}, k=1, \cdots, K-1$ is the column vector with elements $\mathcal{V}_{h,2}^Af(i,k)$, $i=0,1,\cdots,N-1$, and $H_k(\boldsymbol{\Lambda}_A) $ is a diagonal matrix with elements $H_k(\lambda_z), z =1, 2,..., N$. In this scenario, the number of shifted windows, $K$, is not related to the total number of indices $N$.

Based on this formulation of the VGLCT, we can further define two specific vertex-frequency analysis methods: the graph linear canonical wavelet transform (GLWT) and the graph linear canonical stockwell transform (GLST).

\subsubsection{Graph linear canonical wavelet trasform}

In classical signal processing theory, time-frequency analysis and wavelet transform share common objectives. However, these two domains are typically treated independently. Classic time-frequency analysis aims to explore the frequency and time characteristics of non-stationary signals. This approach finds widespread application in signal processing, communication, and audio analysis for tasks like signal analysis, parameter estimation, detection, and denoising. On the other hand, wavelet analysis offers a multi-resolution approach that captures both local and global signal features. Through detailed and approximate decomposition of the signal, wavelet analysis enables tasks such as signal compression, denoising, and feature extraction.

In graph signal processing, it is often assumed that these two signal processing methods are equivalent. Therefore, only the GLWT, which is directly linked to frequency variation (VGLCT), is considered. It can be seen as a specific instance of vertex-frequency analysis, rather than a transform for graph signal compression and its wavelet-like multi-resolution analysis.

\textbf{Linear canonical wavelet operators.} 
Assuming a band-pass filter behavior, the wavelet definition in the spectral domain can be represented by $H(\lambda_z)$ as the basic form. We define the linear canonical wavelet operator at scale $s_j$ as 
\begin{align}
	\label{eq7}
	\widehat{T^{s_j}_Hf}(z)=H(s_j\lambda_z)\hat{f}(z)
\end{align}
where $j=1,\cdots,J$.
This operator is achieved by applying each linear canonical mode to a given function $f$.

Employing the IGLCT on (\ref{eq7}) yields
\begin{align}
	T_H^{s_j}f(z)= \sum_{z=0}^{N-1} \sum_{l=0}^{N-1} H(s_j\lambda_z)\hat{f}(z)
	q_l(n) \kappa_l^{*} p_l^*(z) \epsilon_z^{*}.
\end{align}
The linear canonical wavelets are then realized through localizing it by applying it to the impulse on a single vertex
\begin{align}
	\psi_{s_j,i}(n)
	& =T_H^{s_j}\delta_i(n)
	=\sum_{z=0}^{N-1} \sum_{l=0}^{N-1} H(s_j\lambda_z)
	 \nonumber \\
	&\times \left(\sum_{r=0}^{N-1}\epsilon_k p_r(k) \kappa_r q_r^*(i)\right) q_l(n) \kappa_l^{*} p_l^*(z) \epsilon_z^{*}.
\end{align}

The GLWT is defined as
\begin{align}
	&\mathcal{W}_{\psi}^A(i,s_j)=\sum_{n=0}^{N-1} f(n)\psi_{s_j,i}(n) \nonumber \\
	=&\sum_{n=0}^{N-1} \sum_{p=0}^{N-1} \sum_{l=0}^{N-1} 
	 f(n)  H(s_j\lambda_z)
	\left(\sum_{r=0}^{N-1}\epsilon_k p_r(k) \kappa_r q_r^*(i)\right) \nonumber \\
	&\times q_l(n) \kappa_l^{*} p_l^*(z) \epsilon_z^{*} \nonumber \\
	=&  \sum_{p=0}^{N-1} \sum_{l=0}^{N-1} H(s_j\lambda_z)\hat{f}_A(z) q_l(n) \kappa_l^{*} p_l^*(z) \epsilon_z^{*}.
\end{align}

\begin{example} \label{exam2}
	We use wavelet generation kernels similar to Meyer wavelets \cite{leonardi2011wavelet}, defined as
	\begin{align}
			H(\lambda)= \begin{cases}\sin \left(\frac{\pi}{2} \nu\left(\frac{1}{\lambda_1}|\lambda|-1\right)\right. & \text { if } \lambda_1 \leq \lambda \leq \lambda_2 \\ \cos \left(\frac{\pi}{2} \nu\left(\frac{1}{\lambda_2}|\lambda|-1\right)\right. & \text { if } \lambda_2 \leq \lambda \leq \lambda_3\end{cases}
	\end{align}
	where $v(x)=x^4(35-84x+70x^2-20x^3)$, and $\lambda_1=2/3$, $\lambda_2=2\lambda_1$, $\lambda_3=4\lambda_1$.  The $J$ wavelet scales are defined as $s_j=2^j\lambda_{max}^{-1}$ for $j=1,\cdots,J$, where $\lambda_{max}$ is obtained by the linear caonical Laplacian matrix (\ref{eq8}).
	To handle the low-pass spectral components (the interval for $\lambda$ closest to $ \lambda= 0$), the low-pass type scale function, $G(\lambda)$ is added in the form
	\begin{align}
		G(\lambda)= \begin{cases}1 & \text { if } \lambda \leq \lambda_1 \\ \cos \left(\frac{\pi}{2} \nu\left(\frac{1}{\lambda_1}|\lambda|-1\right)\right. & \text { if } \lambda_1 \leq \lambda \leq \lambda_2\end{cases}.
	\end{align}
	Let $J=4$, the generated linear regular wavelet kernel is shown in Fig. \ref{fig4}. (a). Assuming values of $\sigma=1$, $\xi=0$ and $\beta=0.98$, we can derive four scales: $s_1=1.693, s_2=0.846, s_3=0.423$, and $s_4=0.212$. By centering the wavelet on vertex 90, we can compute it and visualize it in Fig. \ref{fig4} (b)-(e).
\end{example}

\begin{figure}
	\centering
	\subfloat[]{  
		\includegraphics[width=3.5 in]{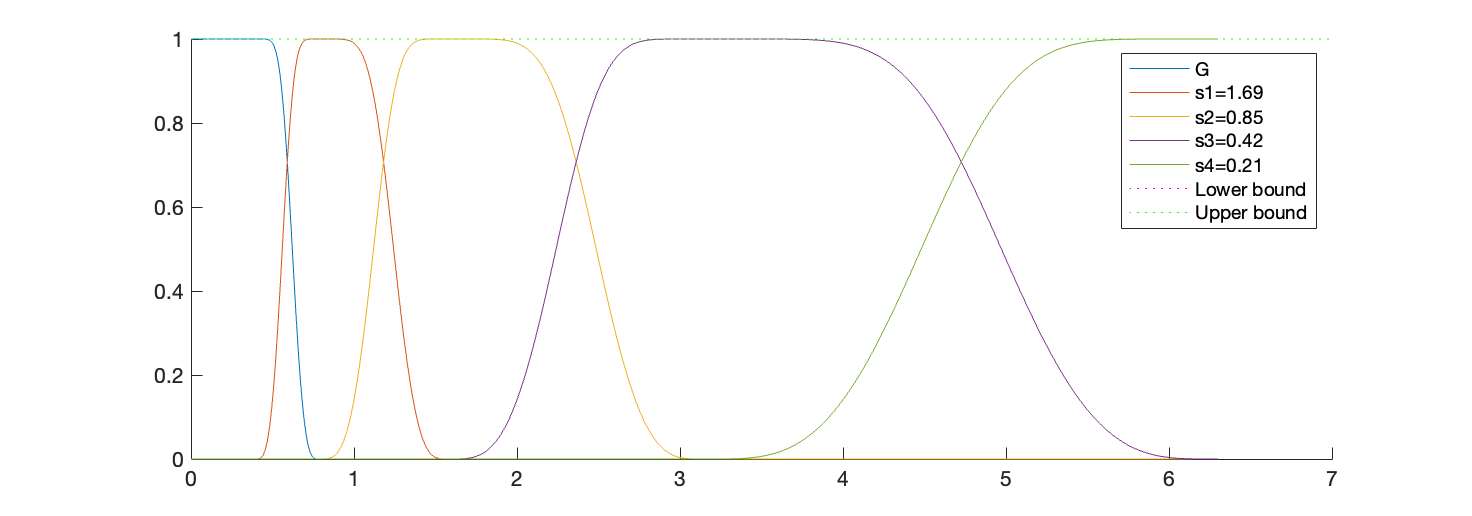}}
	\quad
	\subfloat[]{  
		\includegraphics[width=1.6 in]{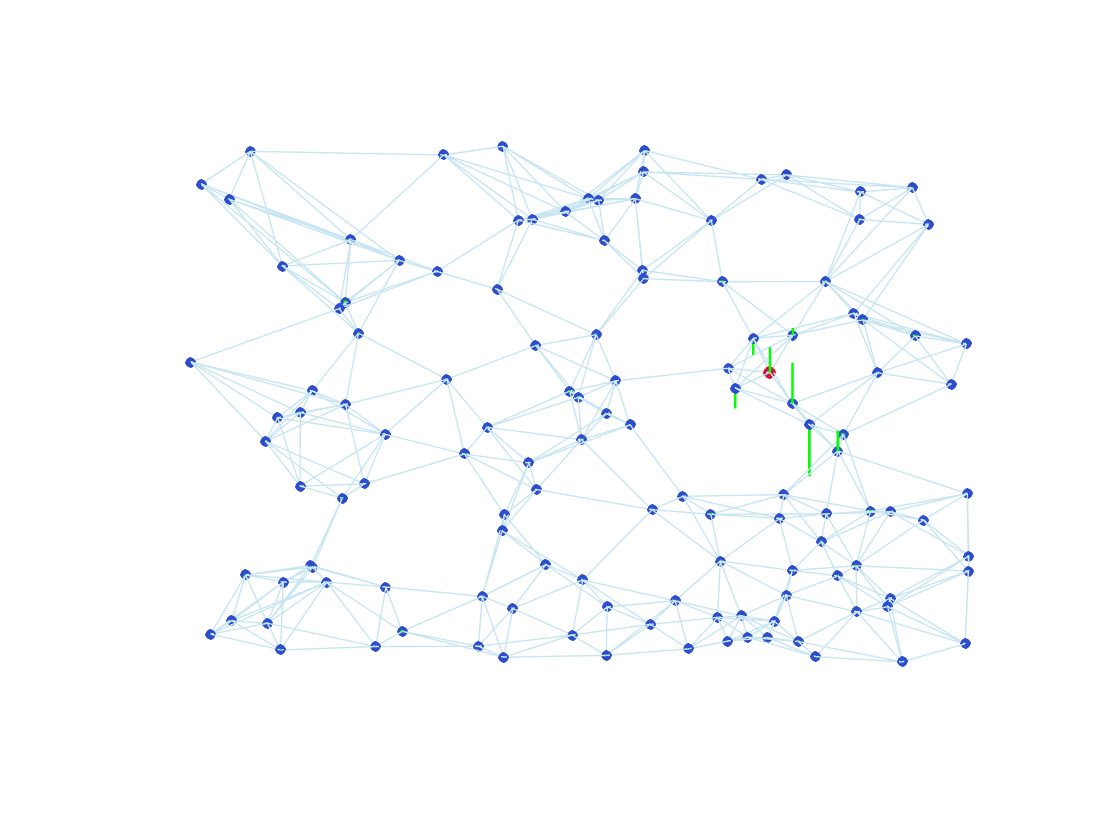}}
	\subfloat[]{
		\includegraphics[width=1.6 in]{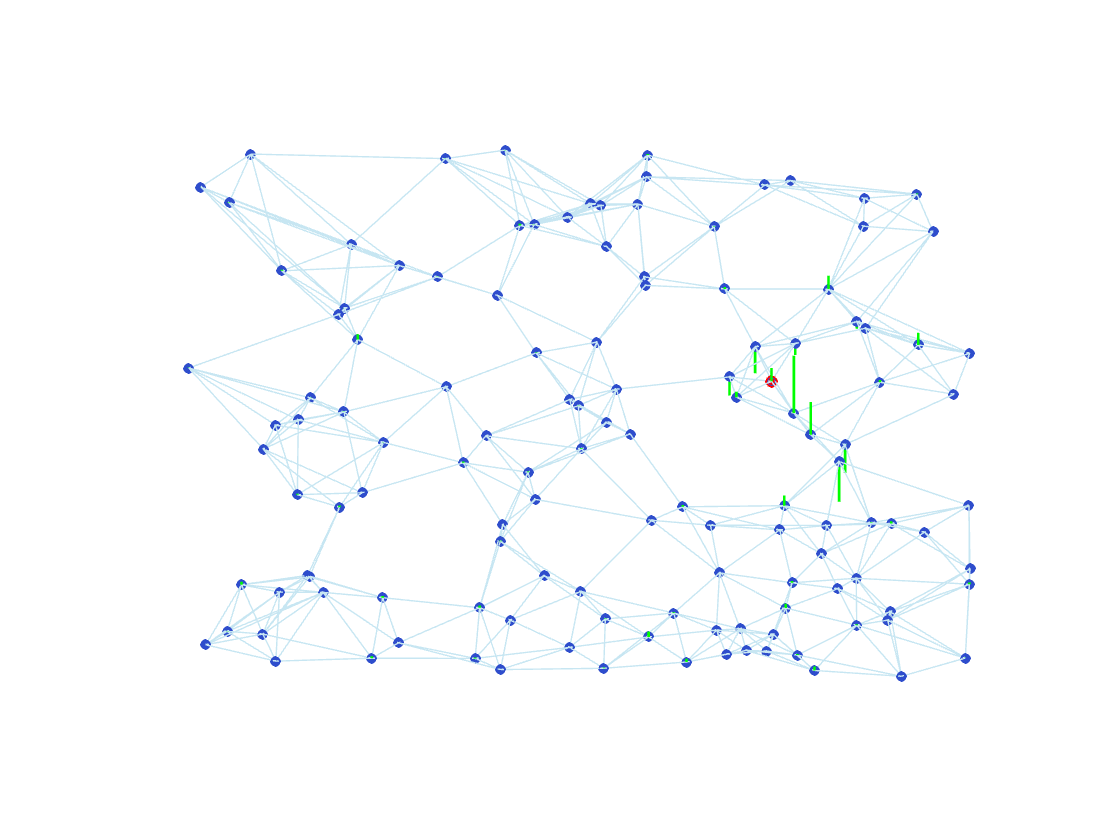}}
	\quad
	\subfloat[]{  
		\includegraphics[width=1.6 in]{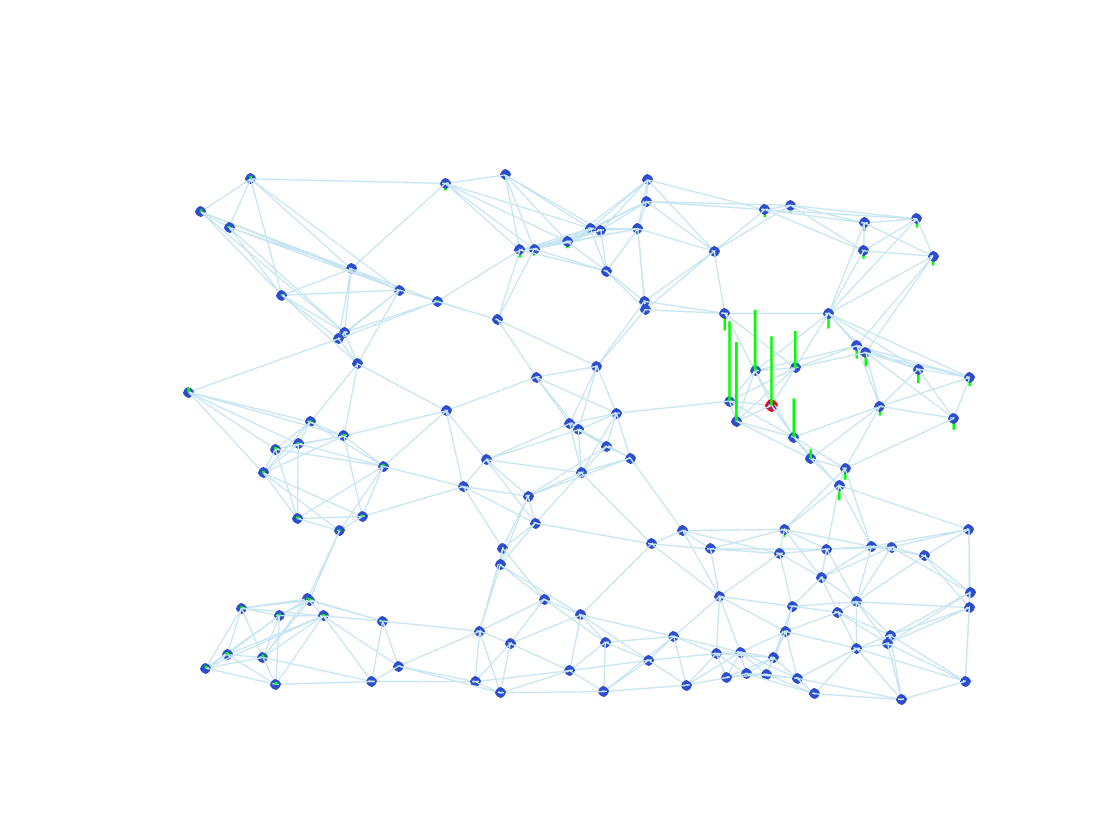}}
	\subfloat[]{
		\includegraphics[width=1.6 in]{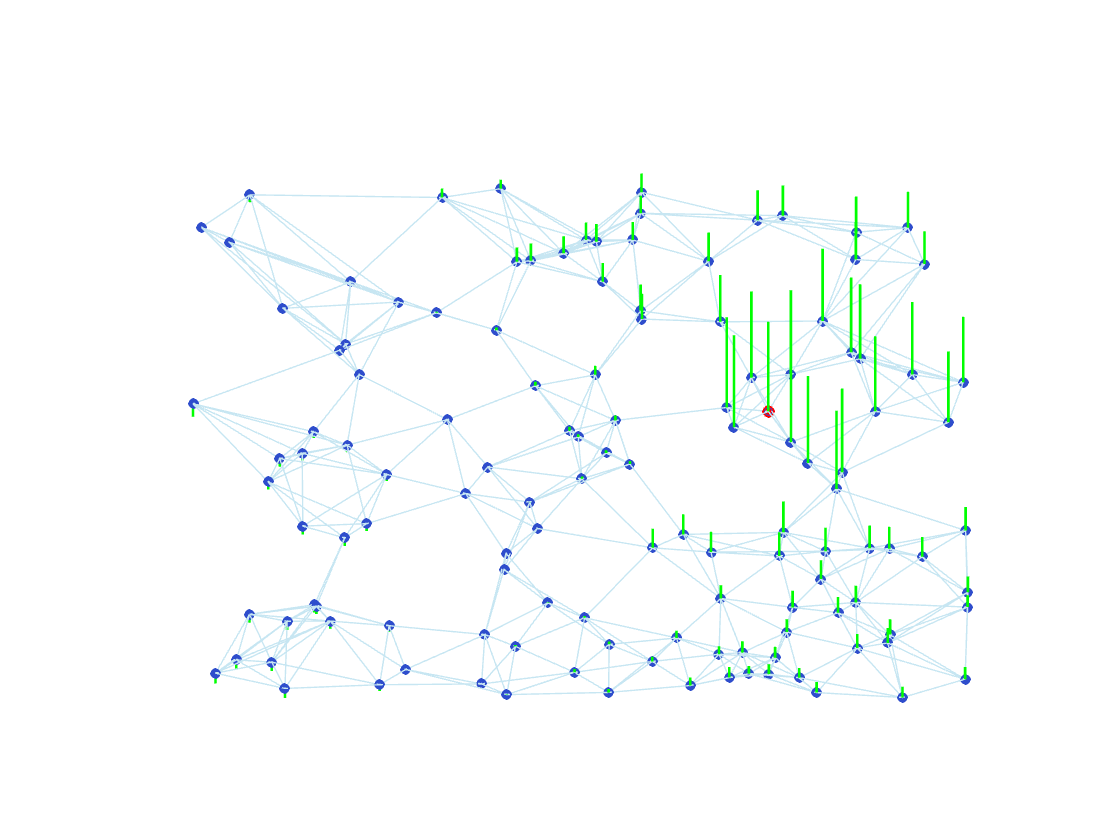}}
	\caption{Graph linear canonical wavelet on Fig. \ref{fig1}. The red dot represents the selected vertex $i=90$, while the green line represents the magnitude of the function value. (a) Wavelet kernels $H(s_j \lambda_z)$. (b) $s_1=1.693, i=90$. (c) $s_2=0.846, i=90$. (d)  $s_3=0.423, i=90$. (e)  $s_4=0.212, i=90$. }\label{fig4}
\end{figure}

\subsubsection{Graph linear canonical Stockwell trasform}
The Stockwell transform is a time-frequency analysis technique that shares similarities with both the wavelet transform and the windowed Fourier transform. The S-transform can be considered as an extension of the WT and WFT, offering a variable time-frequency resolution by using a variable-sized Gaussian window that adapts to the local characteristics of the signal.

\textbf{Graph linear canonical modulation operator.}
For any signal $f(n)$ defined on the graph $\mathcal{G}$ and any $k=0,1,\cdots,N-1$, define the graph linear canonical modulation $M_k^A$ by
\begin{align}
	M_k^Af(n)=\sum_{l=0}^{N-1} f(n)
	q_l(n) \kappa_l^{*} p_l^*(k) \epsilon_k^{*}.
\end{align}

The graph linear canonical Stockwell trasform can be defined as modulation of the GLWT \cite{stockwell1996localization}. For $s=1$, denote its corresponding GLWT as $\mathcal{W}_{\psi}^A(i,s_j)$, the GLST is defined as
\begin{align}
	\mathcal{S}_h^A(i,s_j)
	=& \sum_{l=0}^{N-1} \mathcal{W}_{h}^A(i,s_j)
	q_l(i) \kappa_l^{*} p_l^*(k) \epsilon_k^{*} \nonumber \\
	=& \sum_{l=0}^{N-1}
	\sum_{z=0}^{N-1} \sum_{r=0}^{N-1} H(s_j\lambda_z)\hat{f}_A(p) 
	q_r(n) \kappa_r^{*} p_r^*(z) \epsilon_z^{*} \nonumber \\
	& \times q_l(i) \kappa_l^{*} p_l^*(k) \epsilon_k^{*}.
\end{align}

\begin{example}
	Briefly compare the differences in kernel functions between GLWT and GLST. Utilizing the basis function provided in Example \ref{exam2}, set the linear canonical parameter uniformly to $\sigma=1$, $\xi=0$, $\beta=0.9$ and select vertex 30. The resulting figure illustrates the outcome (Fig. \ref{fig5}).
\end{example}

\begin{figure}
	\centering
	\subfloat[]{  
		\includegraphics[width=1.6 in]{figure/test3_4}}
	\subfloat[]{
		\includegraphics[width=1.6 in]{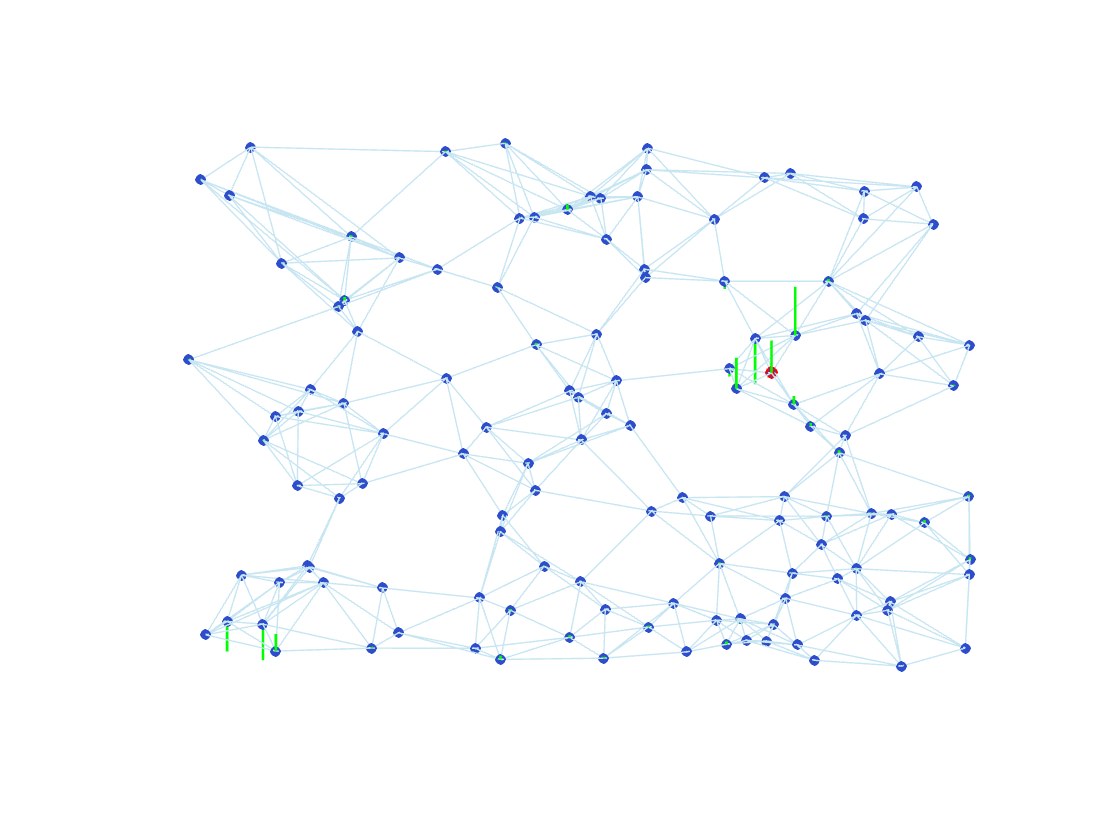}}
	\quad
	\subfloat[]{  
		\includegraphics[width=1.6 in]{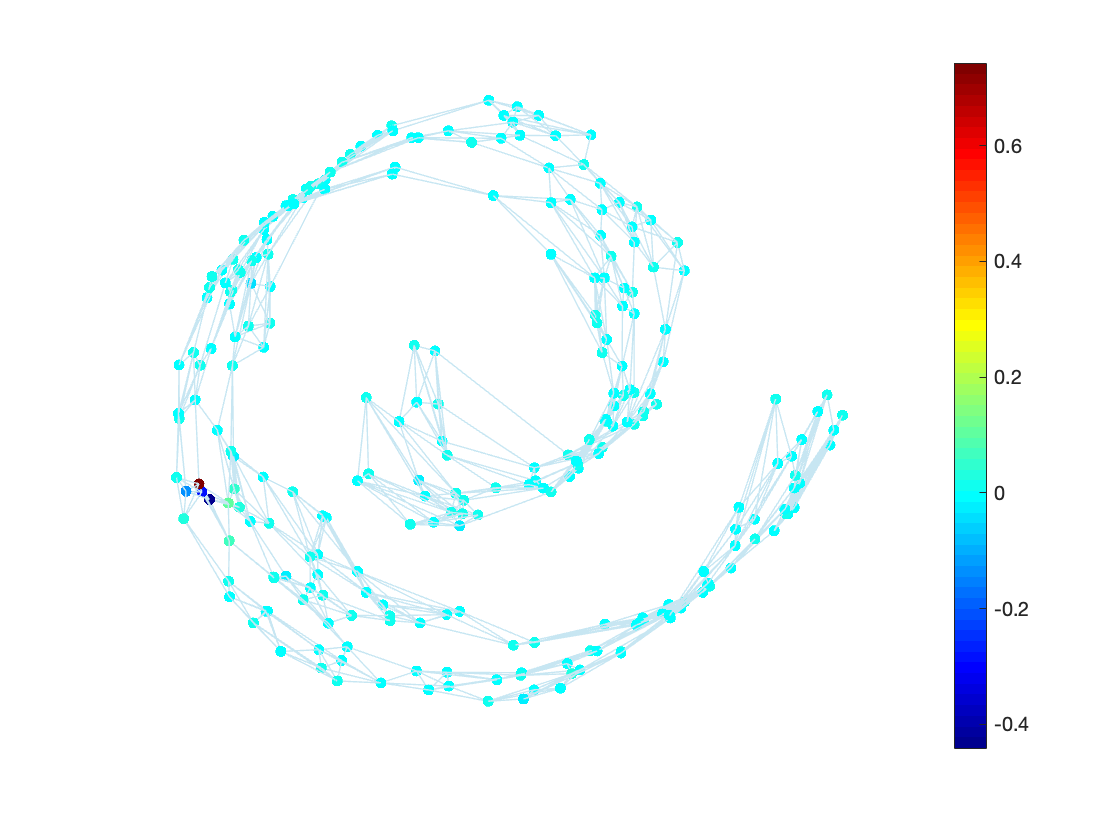}}
	\subfloat[]{
		\includegraphics[width=1.6 in]{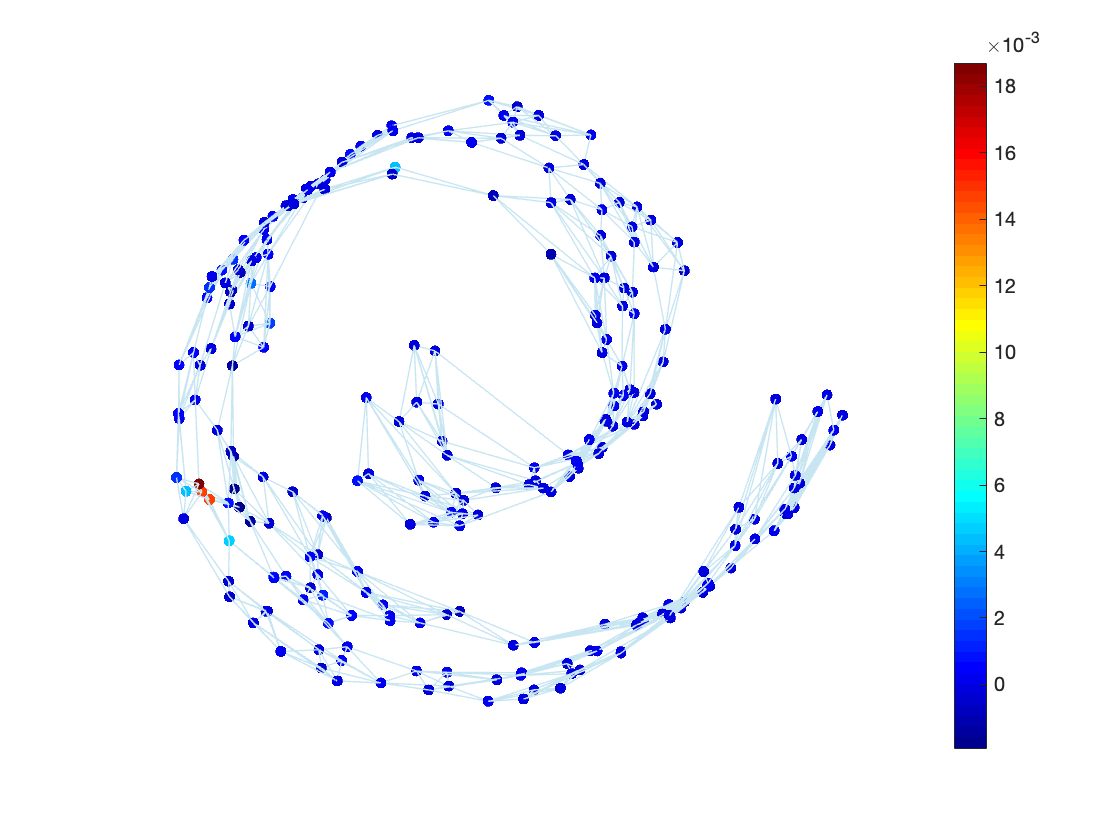}}
	\caption{(a)The GLWT kernel with $s_3=0.423, i=90$ on Fig. \ref{fig1}. (a). (b)The GLST kernel with $s_3=0.423, i=90$ on Fig. \ref{fig1}. (a). (c)The GLWT kernel with $s_3=0.423, i=90$ on Fig. \ref{fig1}. (b). (d)The GLST kernel with $s_3=0.423, i=90$ on Fig. \ref{fig1}. (b). }\label{fig5}
\end{figure}

\begin{example} \label{exam3}
	Considering a path graph with trivial edge weights, we examine the signal $f_3=\sin (60 \pi n / N)$, where $1 \leq n \leq 256$. This signal represents a sampled sinusoidal waveform with a single frequency. 
	Continuing with the utilization of the wavelet from Example \ref{exam2}, we set the parameters as $J=4$, $\sigma=0.5$, $\xi=0$, and $\beta=1.04$. As a result, we obtain four scales: $s_1=1.819$, $s_2=0.909$, $s_3=0.455$, and $s_4=0.227$.
	Figure 6 depicts the GLCT and GLST of the signal $f_3$ at scale $s_3=0.455$.
\end{example}

\begin{figure}
	\centering
	\subfloat[]{  
		\includegraphics[width=1.6 in]{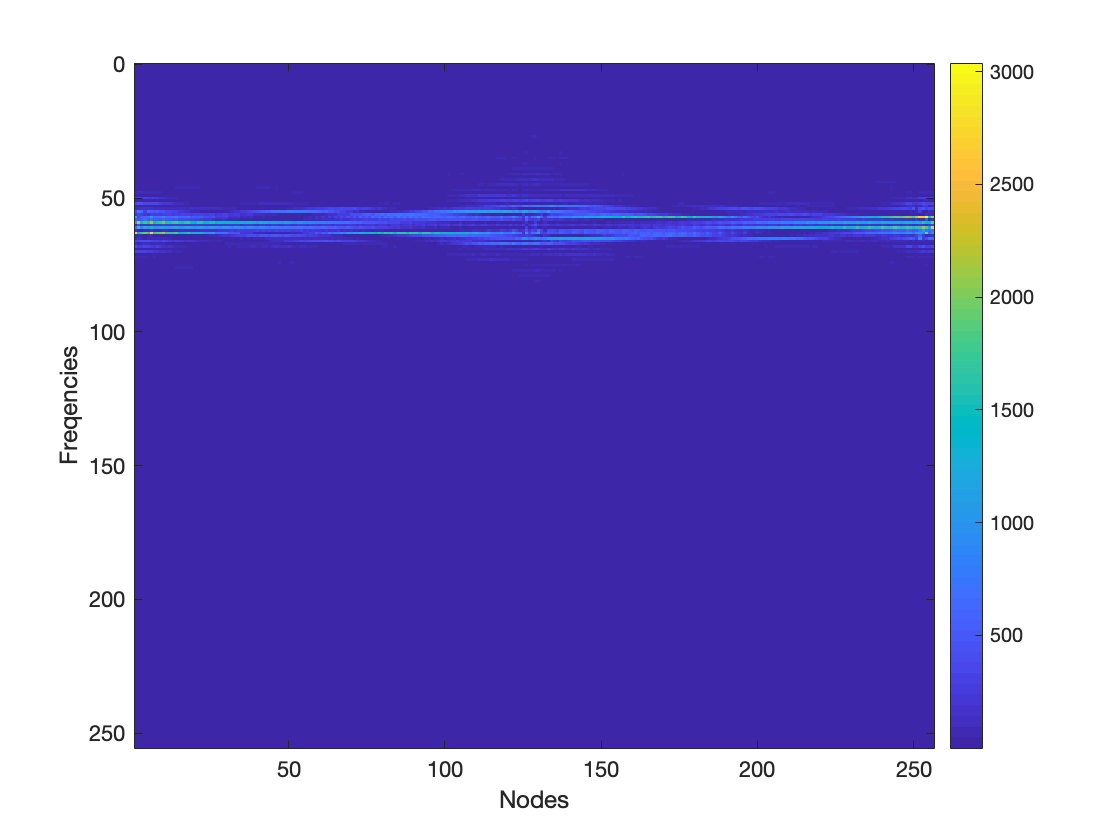}}
	\subfloat[]{
		\includegraphics[width=1.6 in]{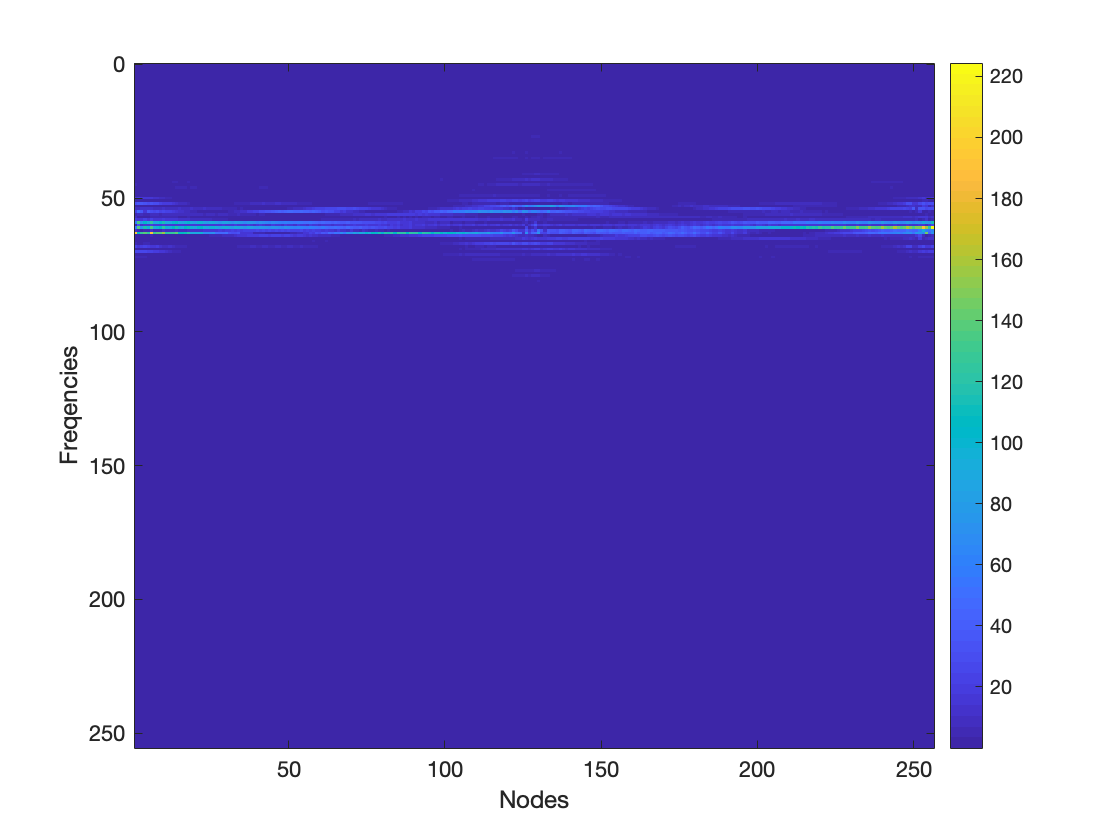}}
	\caption{(a) The GLWT of signal $f_3$. (b) The GLST of signal $f_4$. }\label{fig6}
\end{figure}

\subsubsection{Chebyshev polynomial approximation}
\label{sec3.2.3}
When employing a bandpass function in the windowing process and directly calculating it, conducting spectral analysis on the complete graph becomes imperative for deriving the basis function. To facilitate efficient computation, this section presents a polynomial approximation method grounded in the bandpass function $H_k(\lambda)$. This method enables vertex-frequency analysis using solely local signals and graph connections, thereby enhancing computational speed. 

The discrete points $\lambda=\lambda_z$ are employed for the bandpass function $H_k(\lambda)$. However, since the polynomial approximation is continuous for $0 \leq \lambda \leq \lambda_{max}$, the "min-max" Chebyshev polynomial is naturally selected as it exhibits the least deviation from the desired bandpass function. This choice ensures the accuracy of the approximation.
When the weight function is $\rho(x)=\frac{1}{\sqrt{1-x^2}}$, the orthogonal polynomial obtained by orthogonalizing the sequence $\{1,x,\cdots,x^m,\cdots \}$ is known as the Chebyshev polynomial, which can be represented as $T_m(x)=\cos (m \arccos x)$, where $|x| \leq 1$. 
The Chebyshev polynomial has the following important recursive relationships
\begin{align}
	&T_{m+1}(x)=2xT_m(x)-T_{m-1}(x), \; m=1,2,\cdots \nonumber \\
	&T_0(x)=1, \; T_1(x)=x.
\end{align}

By employing finite $(M-1)$-order Chebyshev polynomials to approximate the spectral domain localization window in (\ref{eq6}), i.e., $H_k(\mathcal{L}_{A}) =\tilde{P}_{k,M-1}(\mathcal{L}_{A}), k=0,\cdots,N-1$. The approximation has the form
\begin{align}
	\tilde{P}_{k,M-1}(\lambda)=\frac{c_{k,0}}{2}+\sum_{m=1}^{M-1}c_{k,m}\tilde{T}_m(\lambda)
\end{align}
where $\tilde{T}_m(\lambda) = \frac{2\lambda}{\lambda_{max}-1}T_m$ is utilized to map the argument from the interval $0 \leq \lambda \leq \lambda_{max}$ to the interval $[-1,1]$.
The Chebyshev coefficients are obtained by the Chebyshev polynomial inversion property, that is
\begin{align} \label{eq17}
	c_{k,m}=\frac{\pi}{2} \int_{-1}^{1} \frac{\tilde{P}_{k,M-1}(\lambda) \tilde{T}_m(\lambda)}{\sqrt{1-\lambda^2}} {\rm d}\lambda.
\end{align}
This implementation is particularly noteworthy when the graph represents large-scale data.

\begin{example} 
	Considering the path graph and utilizing the signal $f_3$ from Example \ref{exam3} along with the wavelet mentioned in Example \ref{exam2}, we proceed with the following parameters: $J=4$, $\sigma=1$, $\xi=0$, and $\beta=0.98$. By applying these parameters, we obtain six scales: $s_1=8.611$, $s_2=4.306$, $s_3=2.153$, $s_4=1.0764$, $s_5=0.538$, and $s_6=0.269$. 
	We obtained six filter kernels at different scales, as depicted in Fig. \ref{fig8}. (a). Subsequently, we approximated these wavelet kernels using a 30th order Chebyshev polynomial. The resulting approximation is presented in Fig. \ref{fig8}. (b). 
	Fig. \ref{fig8}. (c) illustrates the vertex-frequency representation results obtained using the original filter at scale $s_4=1.0764$, while Fig. \ref{fig8}. (d) displays the corresponding results obtained using the approximation filter.
\end{example}

\begin{figure}
	\centering
	\subfloat[]{  
		\includegraphics[width=1.6 in]{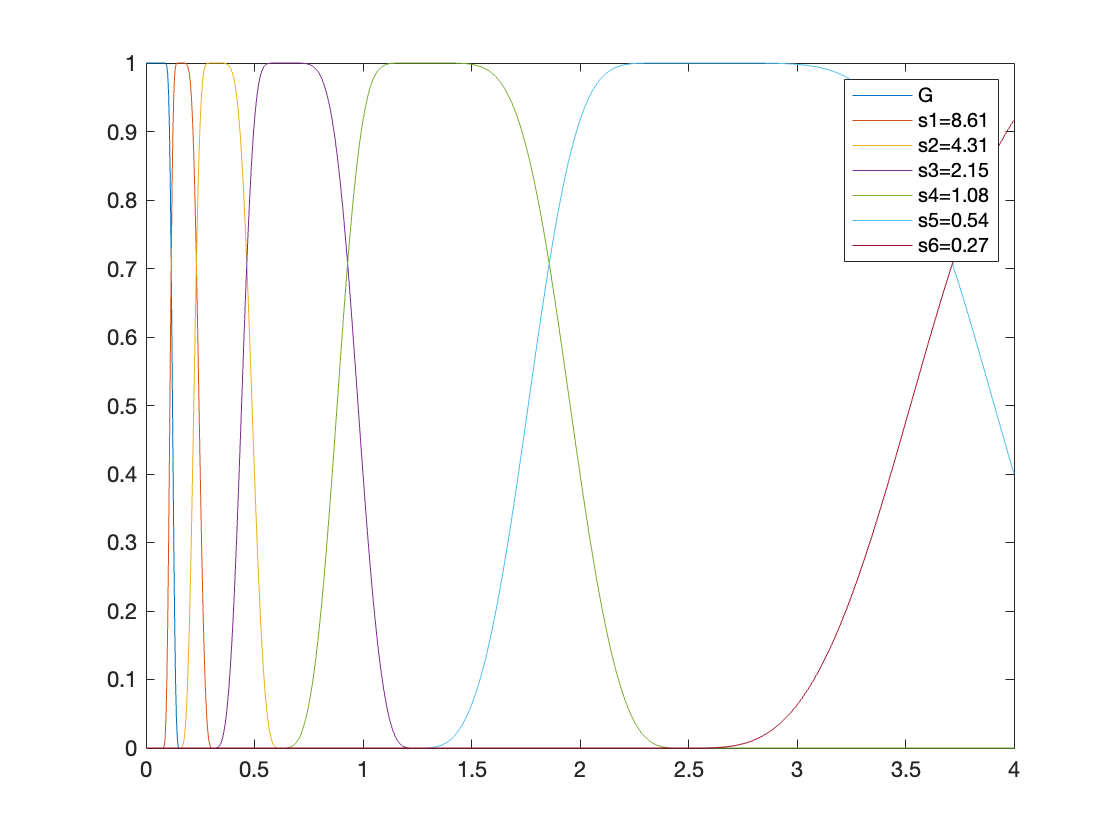}}
	\subfloat[]{
		\includegraphics[width=1.6 in]{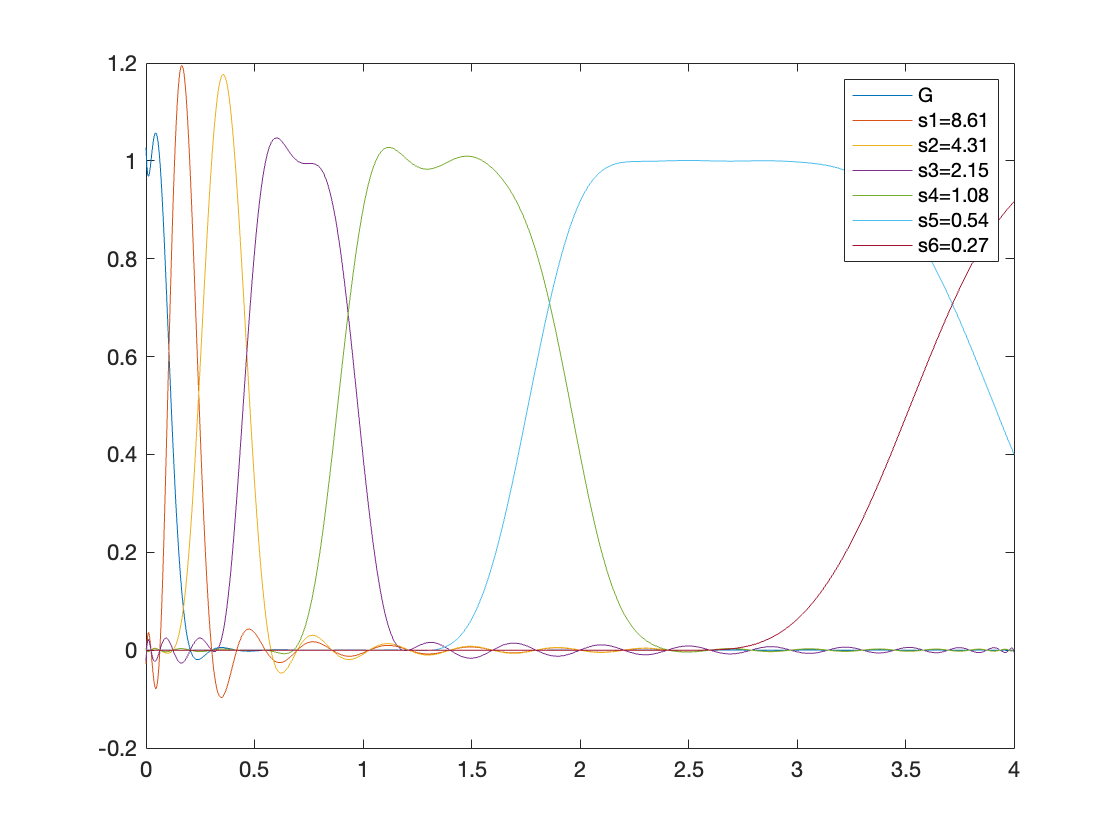}}
	\quad
	\subfloat[]{  
		\includegraphics[width=1.6 in]{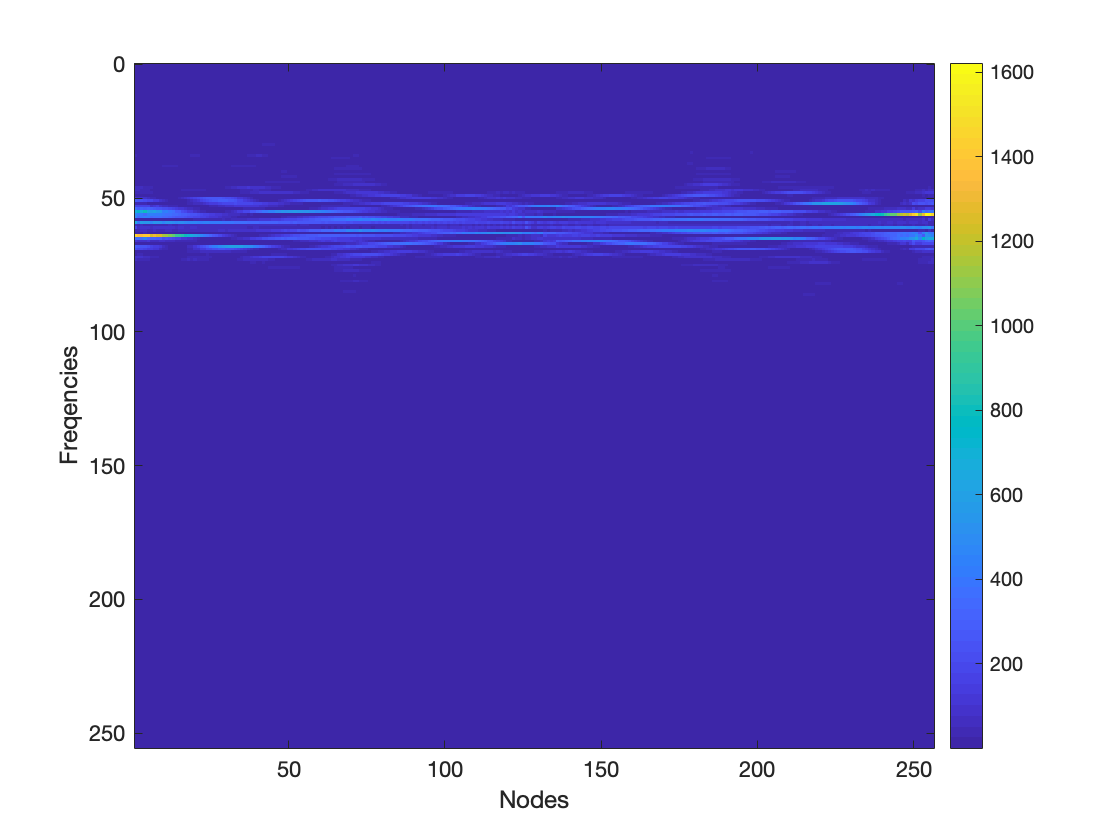}}
	\subfloat[]{
		\includegraphics[width=1.6 in]{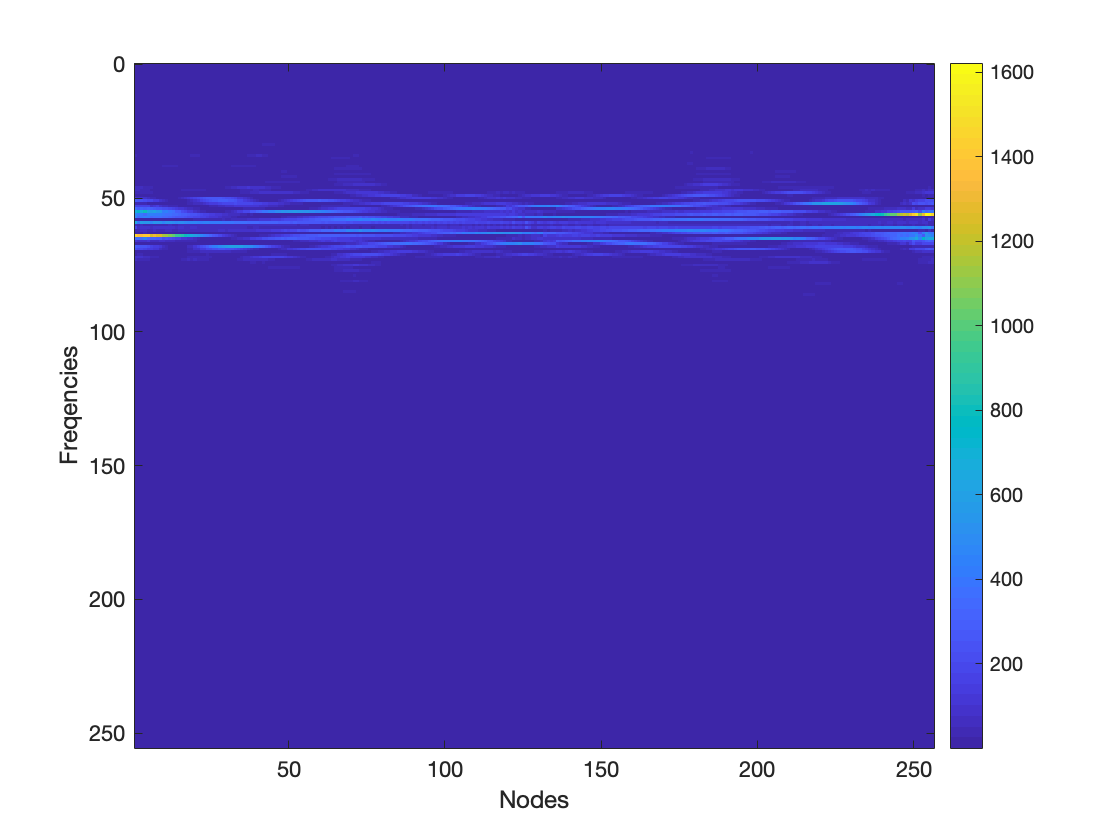}}
	\caption{(a)The GLWT kernel with $s_3=0.423, i=90$ on Fig. \ref{fig1}. (a). (b)The GLST kernel with $s_3=0.423, i=90$ on Fig. \ref{fig1}. (a). (c)The GLWT kernel with $s_3=0.423, i=90$ on Fig. \ref{fig1}. (b). (d)The GLST kernel with $s_3=0.423, i=90$ on Fig. \ref{fig1}. (b). }\label{fig8}
\end{figure}

\subsection{Localized graph linear canonical transform}
The methods mentioned above all operate in the linear canonical domain of the graph. In this section, we discuss a vertex domain localization method denoted as localized graph linear canonical transform (LGLCT).
Let $d_{n,i}$ represent the length of the shortest path from vertex $i$ to vertex $n$. We then define the localized window function as $h(d_{n,i})$, where $h(\cdot)$ represents any basic window function commonly used in classical signal processing. For instance, the Blackman window can be employed \cite{podder2014comparative}, which is defined as follows
\begin{align}
	h(d_{n,i})=0.42-0.5cos\left(\frac{2\pi 
		d_{n,i}}{\eta}\right)+0.08cos\left(\frac{4\pi d_{n,i}}{\eta}\right)
\end{align}
where $\eta$ is the assumed window width.

The distance $d_{n,i}$ can be calculated using the adjacency matrix $\mathbf{A}$. In an unweighted graph, a vertex belonging to a neighborhood of vertex $i$ is represented by a unit value element in row $i$ of the adjacency matrix $\mathbf{A}$. In the case of a weighted graph, the adjacency matrix $\mathbf{A}$ can be derived from the weighted matrix $\mathbf{W}$ by taking the sign of each element, i.e., $\mathbf{A} = {\rm sign}(\mathbf{W})$. Then we define the following matrix form
\begin{align}
	& \mathbf{A}_{d_{n,i}}
	= \begin{cases}\mathbf{A}  & d_{n,i}=1 \\
		(\mathbf{A} \odot \mathbf{A}_{d_{n,i}-1}) \circ (\mathbf{1} - \mathbf{A}_{d_{n,i}-1})\circ (\mathbf{1} - \mathbf{I}) & d_{n,i}\geq 2 \end{cases}  
\end{align}
where $\odot$ represents the logical (Boolean) matrix product, $\circ$ denotes the Hadamard (element-by-element) product, and $\mathbf{1}$ is a matrix with all elements equal to $1$. The $i$-th row of the matrix $(\mathbf{A} \odot \mathbf{A}{d_{n,i}-1})$ provides information about all vertices connected to vertex $i$ through walks of length $K = d_{n,i}$ or lower. It is important to note that the element-wise multiplication of $(\mathbf{A} \odot \mathbf{A}{d_{n,i}-1})$ by the matrix $(\mathbf{1} - \mathbf{A}{d_{n,i}-1})$ removes the vertices connected by walks of length $d_{n,i}-1$, while the multiplication by $(\mathbf{1} - \mathbf{I})$ eliminates the diagonal elements \cite{stankovic2020vertex}.

For a window function $\mathbf{h}=[h(0),\cdots,h(N-1)]^T \in \mathbb{R}^N$, the localized window function (denoted as $\mathbf{T}_{\mathbf{A}}\mathbf{h}=\left[T_{\mathbf{A}}h(n,i)\right]_{N-1\times N-1}$) with assumed graph window width $\eta$ can be given by
\begin{align}
	\mathbf{T}_{\mathbf{A}}\mathbf{h}=h(0)\mathbf{I}+h(1)\mathbf{A}_1+\cdots+h(\eta-1)\mathbf{A}_{\eta-1}.
\end{align}
The LGLCT of a graph signal $f(n)$, then becomes
\begin{align}
	\mathcal{V}_{h,3}^Af(i,k)= \sum_{n=0}^{N-1} \sum_{l=0}^{N-1} f(n) T_{\mathbf{A}}h(n,i) \epsilon_k p_l(k) \kappa_l q_l^*(n).
\end{align}

\begin{example}
	Given a path graph with trivial edge weights, we consider
	the following two time series: $f_4=\begin{cases}\sin (60 \pi n / N), & 1 \leq n \leq 128 \\ \sin (90 \pi n / N), & 129<n \leq 256\end{cases}$ is a sampled sinusoidal signal
	with two different frequencies, $f_5(n)=\sin \left(\left(20 n+0.5 n^2\right) \pi / N\right), 1 \leq n \leq 256$ is a chirp signal. 
	The LGLCT of signals $f_4$ and $f_5$ is depicted in Fig. \ref{fig7}.
\end{example}

\begin{figure}
	\centering
	\subfloat[]{  
		\includegraphics[width=1.6 in]{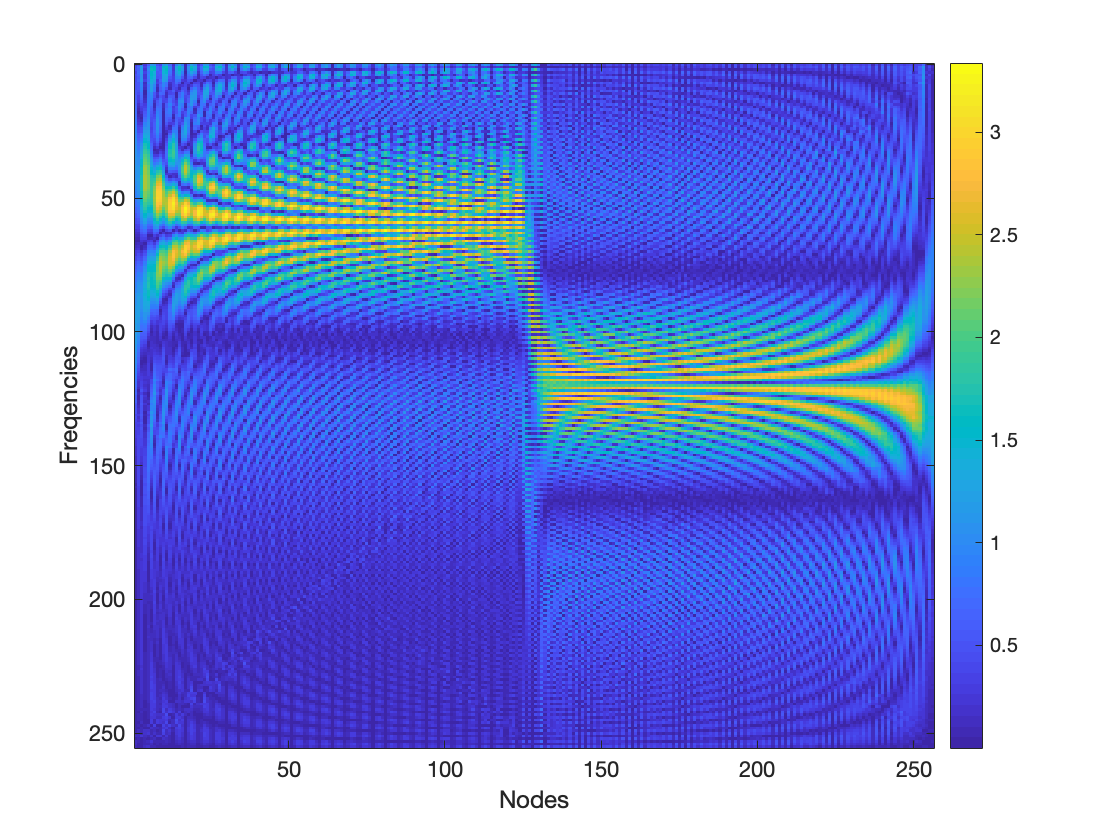}}
	\subfloat[]{
		\includegraphics[width=1.6 in]{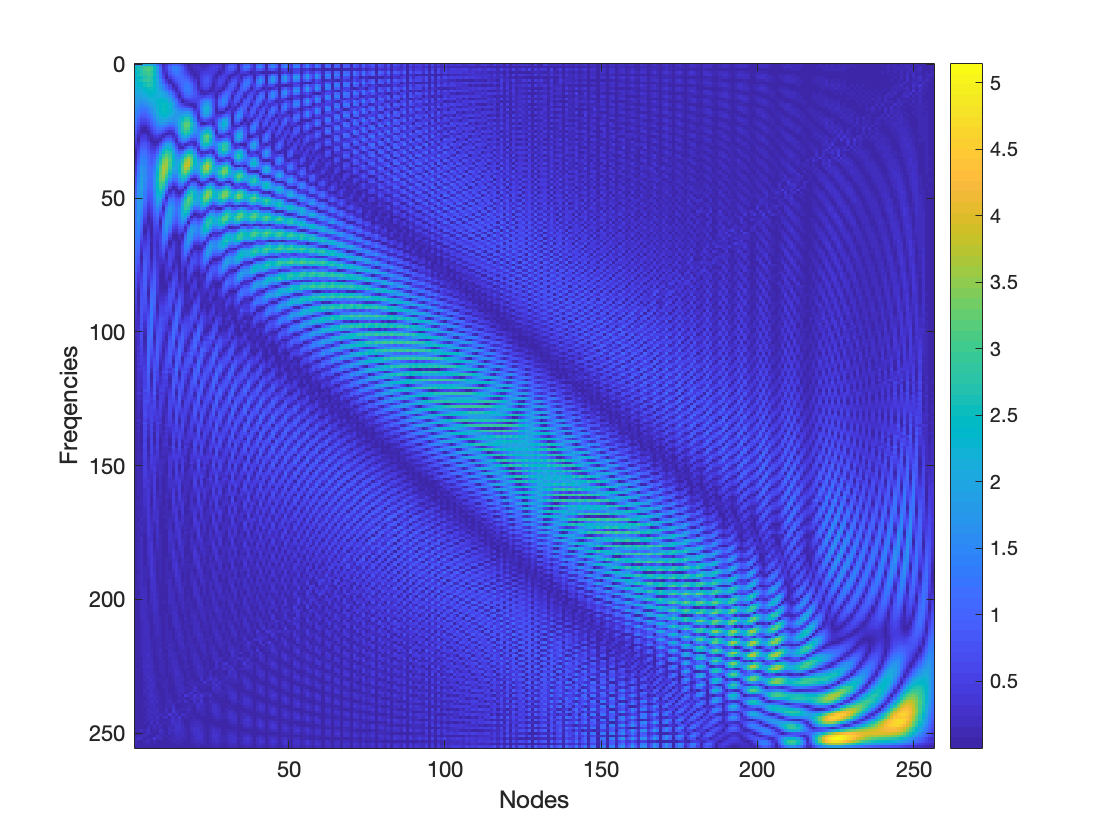}}
	\caption{The LGLCT of graph signals. (a) The LGLCT of signal $f_4$. (b) The LGLCT of signal $f_5$. }\label{fig7}
\end{figure}

\section{Filter Design} \label{sec5}
Within a basic veryex-frequency analysis framework, different filter selections can correspond to various transformations. The selection of an appropriate filter is crucial in practical applications. For instance, different spectral graph neural networks mainly rely on distinct filter learning approaches \cite{chen2020understanding, bo2023survey, wang2022powerful}. Therefore, in this section, we will initially explore optimization learning methods for filters and subsequently apply them to image classification tasks.

\subsection{Optimal Filter}
In Section \ref{sec3.1}, the graph linear canonical convolution operator was presented as $\mathbf{f} *_A \mathbf{h}= \mathcal{F}_{A}^H \left( \mathcal{F}_{A}\mathbf{f} \odot \mathcal{F}_{A}\mathbf{h} \right)$, as shown in (\ref{eq9}), where $\odot$ reprents a point-wise product.
In general, designing filters on a graph involves designing a diagonal matrix $\mathbf{H}(\boldsymbol{\Lambda}_A)={\rm diag}(\mathbf{\hat{h}})$, where $\mathbf{\hat{h}}=[h_1,h_2,\cdots,h_N]^{\top}$. Then, the graph linear canonical convolution operator can be rewritten as
\begin{align}
	\mathbf{f} *_A \mathbf{h}= \mathcal{F}_{A}^H \mathbf{H}(\boldsymbol{\Lambda}_A) \mathcal{F}_{A}\mathbf{f} =\mathbf{H}(\mathcal{L}_{A}) \mathbf{f}.
\end{align}
It follows that the signal $\mathbf{f}$ is filtered by $\mathbf{H}(\mathcal{L}_{A})$.

\begin{remark}
	Extending convolution to inputs $\mathbf{f}$ with multiple input channels is a straightforward process. If $\mathbf{f}$ is a signal with $M$ input channels and $N$ locations, we apply the GLCT $\mathcal{F}_{A}$ to each channel independently. Then, to perform the convolution, we utilize multipliers $\mathbf{\hat{h}} = (h_{i,j} ; i \leq N , j \leq M)$.
\end{remark}

Consider an input signal $\mathbf{f}$, where the actual output signal is represented as
\begin{align}
	\mathbf{g}=\mathbf{Gf}+\mathbf{n},
\end{align}
where $\mathbf{G}$ is a known matrix, $\mathbf{n}$ is the additive noise term, and $\mathbf{f}$ is a stochastic graph signal \cite{ozturk2021optimal}.

Our next objective is to learn an optimized filter that minimizes the loss function while predicting the output signal $\widetilde{\mathbf{g}}$. We define the squared loss function as
$R(\mathbf{H}) = \frac{1}{N} \|\mathbf{H}(\mathcal{L}_{A}) \mathbf{f} - \mathbf{g}\|_F^2$.
Therefore, our objective is to design the matrix $\mathbf{H}(\mathcal{L}_{A})$ to solve the following optimization problem
\begin{align} \label{eq10}
	\min_{\mathbf{H}}
	\quad
	\frac{1}{N} \|\mathbf{H}(\mathcal{L}_{A}) \mathbf{f} - \mathbf{g}\|_F^2
\end{align}
for any $A$. Subsequently, we will explore different possible values of $A$ to identify $A^*$ that minimizes the loss.

\begin{thm}
The problem defined in (\ref{eq10}) will yield identical loss values for any $A$ if there are no constraints on $\mathbf{H}(\mathcal{L}_{A})$.
\end{thm}
\begin{proof}
	Consider any parameters $A_1$ and $A_2$ such that $A_1 \neq A_2$, and define
	\begin{align} \label{eq11}
		\mathbf{H}(\boldsymbol{\Lambda}_{A_j})=\arg \min _{\mathbf{H}} \left\{\left\|\mathbf{H}(\mathcal{L}_{A_j}) \mathbf{f} - \mathbf{g}\right\|_F^2\right\} 
	\end{align}
	where $j=1,2$. 
	We note that for any choice of $\mathbf{H}(\boldsymbol{\Lambda}_{A_j})$, according to $j=1$ in (\ref{eq11}), we have
	\begin{align}
		\left\|\mathcal{F}_{A_1}^H \mathbf{H}(\boldsymbol{\Lambda}_{A_1}) \mathcal{F}_{A_1} \mathbf{f} - \mathbf{g}\right\|_F^2
		\leq
		\left\|\mathcal{F}_{A_1}^H \mathbf{H}(\boldsymbol{\Lambda}) \mathcal{F}_{A_1} \mathbf{f} - \mathbf{g}\right\|_F^2. \nonumber
	\end{align}
	Then, by choosing $\mathbf{H}(\boldsymbol{\Lambda})= \mathcal{F}_{A_1} \mathbf{H}(\mathcal{L}_{A_2}) \mathcal{F}_{A_1}^H$, we can obtain from $j=2$ in (\ref{eq11}) that
	\begin{align}
		\left\|\mathcal{F}_{A_1}^H \mathbf{H}(\boldsymbol{\Lambda}_{A_1}) \mathcal{F}_{A_1} \mathbf{f} - \mathbf{g}\right\|_F^2
		\leq
			\left\|\mathcal{F}_{A_2}^H \mathbf{H}(\boldsymbol{\Lambda}_{A_2}) \mathcal{F}_{A_2} \mathbf{f} - \mathbf{g}\right\|_F^2. \nonumber
	\end{align}
	Similarly, it can be easily shown that
	\begin{align}
		\left\|\mathcal{F}_{A_2}^H \mathbf{H}(\boldsymbol{\Lambda}_{A_2}) \mathcal{F}_{A_2} \mathbf{f} - \mathbf{g}\right\|_F^2
		\leq
		\left\|\mathcal{F}_{A_1}^H \mathbf{H}(\boldsymbol{\Lambda}_{A_1}) \mathcal{F}_{A_1} \mathbf{f} - \mathbf{g}\right\|_F^2. \nonumber
	\end{align}
Therefore, we can conclude that the same minimum value is achieved for any $A_1$ and $A_2$.
\end{proof}

By deriving the objective function in (\ref{eq11}), we can identify the optimal solution. The subsequent theorem delineates the solution to the optimal filtering problem stated in (\ref{eq11}).

\begin{thm} \label{thm2}
	For the problem in (\ref{eq11}), the optimal filter coefficients $\mathbf{\hat{h}}^{\text{opt}}=[h_1^{\text{opt}},\cdots, h_N^{\text{opt}}]^{\top}$ are obtained from the following matrix equation
	\begin{align}
		(\mathcal{F}_{A} \mathbf{f} \odot \mathcal{F}_{A} \mathbf{f}) \mathbf{1}_N \odot  \mathbf{\hat{h}}^{\text{opt}} = (\mathcal{F}_{A} \mathbf{f} \odot  \mathcal{F}_{A} \mathbf{g}) \mathbf{1}_N.
	\end{align}
\end{thm}
\begin{proof}
	Consider a fixed parameter $A$, the squared loss function is expressed as
	$R = \frac{1}{N} \|\mathcal{F}_{A}^H \mathbf{H}(\boldsymbol{\Lambda}) \mathcal{F}_{A} \mathbf{f} - \mathbf{g}\|_F^2$.
	We rewrite it in the spectral domain as
	\begin{align} \label{eq16}
		\hat{R} = \frac{1}{N} \|\mathbf{H}(\boldsymbol{\Lambda}) \mathcal{F}_{A} \mathbf{f} -  \mathcal{F}_{A}\mathbf{g}\|_F^2.
	\end{align}
	Taking the derivative of $\hat{R}$ with respect to $\mathbf{\hat{h}}$, we obtain
	\begin{align}
		\nabla_{\mathbf{\hat{h}}} \hat{R} = \frac{2}{N} \left(\mathcal{F}_{A} \mathbf{f} \odot (\mathbf{\hat{h}} \odot \mathcal{F}_{A} \mathbf{f} - \mathcal{F}_{A} \mathbf{g} ) \right) \mathbf{1}_N.
	\end{align}
	Let $\nabla_{\mathbf{\hat{h}}} \hat{R} = 0$ and $\mathbf{\hat{h}}^{\text{opt}}=[h_1^{\text{opt}},\cdots, h_N^{\text{opt}}]^{\top}$, we obtain
	\begin{align}
		(\mathcal{F}_{A} \mathbf{f} \odot \mathcal{F}_{A} \mathbf{f}) \mathbf{1}_N \odot  \mathbf{\hat{h}}^{\text{opt}} = (\mathcal{F}_{A}\mathbf{f} \odot  \mathcal{F}_{A} \mathbf{g}) \mathbf{1}_N.
	\end{align}
\end{proof}

\begin{example} \label{exam7}
	Consider the graph structure depicted in Example \ref{exam6}, where the spectral content of the graph signal length on each vertex follows a normal distribution, i.e., $\mathbf{\hat{f}}_6(k) \sim \mathcal{N}(e^{-1.5*\lambda_k}, (e^{-1.5*\lambda_k}/2)^2)$. Thus, the graph signal $\mathbf{f}_6$ is a matrix of size 100 × 200, where $\mathbf{f}_6(n)=\mathcal{F}_{A}^H\mathbf{\hat{f}}_6(k)$ denotes the signal on the nth vertex, a vector with a length of 200.
	Let the parameters of GLCT be set to $\sigma=1.5$, $\xi=0.5$, $\beta=0.95$. The ideal filter is defined as $h_k=e^{-2 \lambda_k}$, where $k=0,\cdots,N-1$, and the noise term $\mathbf{n}$ follows a normal distribution $\mathbf{n} \sim \mathcal{N}(0, 0.1^2)$. Consequently, the actual output signal is given by $\mathbf{g}=\mathcal{F}_{A}^H \mathbf{H(\Lambda)} \mathcal{F}_{A} \mathbf{f}_6+\mathbf{n}$.
	According to Theorem \ref{thm2}, the optimal filter design result can be derived. Figure \ref{fig10} illustrates the obtained optimal filter result depicted by the orange curve, while the blue curve represents the ideal filter.
\end{example}

\begin{figure}
	\centering
	\includegraphics[width=0.5\textwidth]{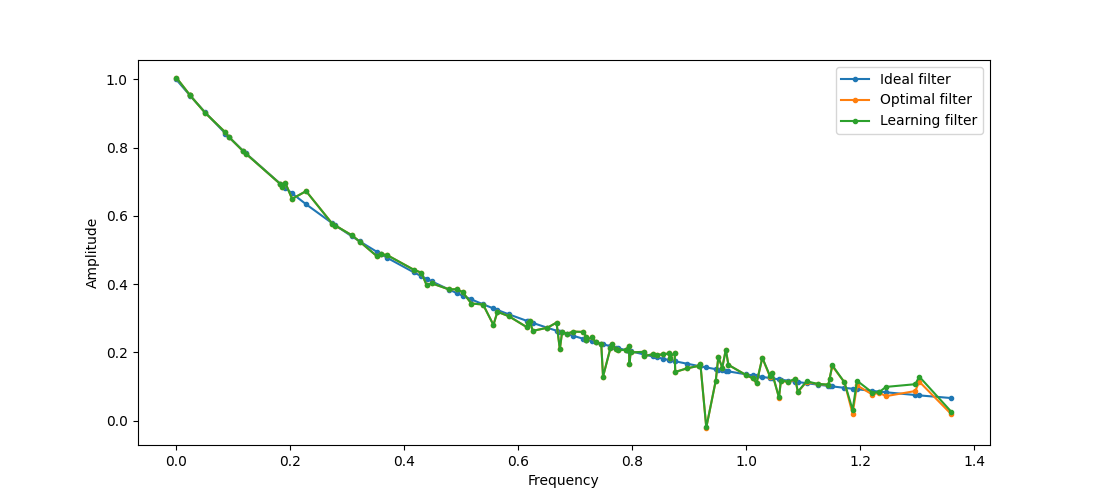}
	\caption{Ideal Filter, optimal filter and learning filter. }\label{fig10}
\end{figure}

\subsection{Learning Filter}
In the aforementioned optimal filter design method, the gradient of the loss function constitutes a non-convex function, rendering it challenging to guarantee the discovery of the global optimal solution.
To accommodate diverse datasets and problem scenarios more flexibly, optimization algorithms like gradient descent can be employed for filter learning. This is particularly pertinent when dealing with more intricate models or elevated performance criteria. 
Furthermore, alongside the fundamental gradient descent method, numerous enhanced optimization algorithms exist, including stochastic gradient descent (SGD), the momentum method, Adam, Adagrad, and others. These algorithms offer various advantages tailored to different problem domains, aiding in expediting convergence and enhancing algorithmic performance.

Here, we introduce a method for learning filters via SGD, which serves as the groundwork for future advancements in spectrogram neural networks, as shown in Algorithm \ref{algo2}. 

\begin{algorithm}[!h]
	\label{algo2}
	\caption{SGD-based filter learning algorithm}
	\begin{algorithmic}[1]
		\REQUIRE ~~\\
		Input graph signal: $\mathbf{f}$,  Target signal: $\mathbf{g}$.\\
		Initial value of filter: $\mathbf{H}^0$. \\
		The parameters of GLCT: $A=(\xi, \sigma, \beta)$. \\
		Learning rate: $\epsilon$, Size of the mini-batch: $N_{\text{mini}}$, Stopping criterion, Maximum number of iterations.\\
		\ENSURE ~~\\
		The learned filter: $\mathbf{H}^{\text{sgd}}={\rm diag}(\mathbf{\hat{h}}^{\text{sgd}})$.\\
		Loss function value: $\hat{R}(\mathbf{H}^{\text{sgd}})$.
		
		\STATE Compute $\mathcal{F}_{A} =  \Upsilon \mathbf{P} \mathcal{K} \mathbf{Q}^H$.
		\STATE Compute the loss function in (\ref{eq16}) of the initial filter: $\hat{R}(\mathbf{H}^0)$.
		\STATE Let $\mathbf{H}^{\text{sgd}}=\mathbf{H}^0$.
		\STATE \textbf{while} the stopping criteria are not met \textbf{do} :\\
		       Retrieve a mini-batch of samples with a length of $N_{\text{mini}}$ from the training set. \\
		        Calculate the gradient of the loss function:     
		        $\nabla_{\mathbf{\hat{h}}^{\text{sgd}}} \hat{R} = \frac{2}{N} \left(\mathcal{F}_{A} \mathbf{f} \odot (\mathbf{\hat{h}}^{\text{sgd}} \odot \mathcal{F}_{A} \mathbf{f} - \mathcal{F}_{A} \mathbf{g} ) \right) \mathbf{1}_N$. \\
		        Filter update: $\hat{\mathbf{h}}^{\text{sgd}}=\hat{\mathbf{h}}^{\text{sgd}}-\epsilon \nabla_{\mathbf{\hat{h}}^{\text{sgd}}} \hat{R} $. \\
		         \textbf{end while}
		 \STATE Compute the loss function in (\ref{eq16}) of the learned filter: $\hat{R}(\mathbf{H}^{\text{sgd}})$.
	\end{algorithmic}
\end{algorithm}

\begin{example}
	Continuing with the consideration of the graph signal $\mathbf{f}$ and target signal $\mathbf{g}$ in Example \ref{exam7}, we set the initial filter to a random sequence following a standard distribution, i.e., $\hat{\mathbf{h}}^0 \sim \mathcal{N}(0, 1)$.
	The parameters of GLCT are set to $\sigma=1.5$, $\xi=0.5$, and $\beta=0.95$, with a learning rate of $\epsilon=0.6$, a mini-batch size of $N_{\text{mini}}=100$, a stopping criterion of $0.001$, and a maximum number of iterations of $100$. 
	The final learned filter is depicted in Figure 10, where the green curve represents the result.
\end{example}

\subsection{Image Classification based on designed filter}
In addition to the direct utilization of advanced filters \cite{bruna2013spectral, liao2019lanczosnet}, many researchers also employ polynomial form filters, which are extended in polynomial form (e.g., Chebyshev polynomial expansion as discussed in Section \ref{sec3.2.3}). Polynomial filters offer the advantages of achieving locality and reducing learning complexity. Several notable examples include CheyNet \cite{defferrard2016convolutional}, GPRGNN \cite{chien2020adaptive}, JacobiConv \cite{wang2022powerful}, among others.

In this section, we employ the filter representation of Chebyshev polynomials. Subsequently, leveraging this filter formulation, we construct a fundamental linear learner for image classification described as $\mathbf{H(\Lambda)}
	=\sum_{m=0}^{M-1}c_{k,m}\tilde{T}_m(\boldsymbol{\Lambda})$,
where $c_{k,m}$ represents the polynomial coefficient, as elucidated in (\ref{eq17}). This coefficient ensemble forms the essential parameter set for training the learner layer. 
To augment the fitting capability of the learner, a parameterized weight matrix $\Theta$ is introduced to perform radial transformations on the input graph signal matrix, yielding 
$\widetilde{\mathbf{g}} = \mathcal{F}_{A}^H \mathbf{H}(\boldsymbol{\Lambda}) \mathcal{F}_{A}\mathbf{f} \Theta$.
The squared loss function in spectral domain is expressed as
$\hat{R} = \frac{1}{N} \|\mathbf{H}(\boldsymbol{\Lambda}) \mathcal{F}_{A} \mathbf{f}\Theta -  \mathcal{F}_{A}\mathbf{g}\|_F^2$.
Subsequently, the filter undergoes training using Algorithm \ref{algo2} to address image classification tasks.

To assess the validity of our model, we applied it to the Euclidean scenario using the benchmark MNIST classification problem \cite{defferrard2016convolutional}, which comprises a dataset of 70,000 digits depicted on a 2D grid measuring 28 × 28. We specifically focused on discriminating between the challenging digits 4 and 9, each comprising 6,824 samples, totaling 13,648 instances.
For our graph-based model, we constructed a 4-NN graph from the 2D grid, resulting in a graph with $N = \|\mathcal{V}\| = 784$ nodes. The weights of the K-NN similarity graph were computed using 
$w_{i,j}=\exp \left(-\frac{\left\|z_i-z_j\right\|_2^2}{\epsilon^2}\right)$,
where $z_i$ denotes the 2D coordinate of pixel $i$.
For the dataset comprising 13,648 samples, 70\% were randomly selected for training, while the remaining 30\% were allocated to the test set. 

\begin{example}
	Let the parameters of GLCT be $\sigma=0.5$, $\xi=0$, and $\beta=1.1$, with a learning rate of $\epsilon=0.001$. The degree of the Chebyshev polynomial is $M=10$.
	The accuracy achieved on the training set is 95.15\%, while the accuracy on the test set is 93.58\%.
\end{example}

\begin{example}
	We conducted tests on the varying outcomes of the linear canonical parameter amidst minor fluctuations, while keeping all other parameters constant (change $\sigma$ and $\beta$, , keeping $\xi = 0$ fixed). As depicted in Table \ref{tab1}, the highest accuracy was attained when the parameters were $\sigma=0.9$, $\xi=0$, and $\beta=0.5$. This suggests that fine-tuning the linear canonical parameter can enhance model performance.
\end{example}

\begin{table*}[h]
	\centering
	\caption{Prediction Results of Test Set with Different Parameters (\%)}
	\label{tab1}  
	\begin{tabular}{l l l l l l l l l l l}
		\hline\hline\noalign{\smallskip}	
		 &  $\sigma=0.5$  & $\sigma=0.6$ & $\sigma=0.7$ & $\sigma=0.8$ & $\sigma=0.9$ & $\sigma=1$ & $\sigma=1.1$ & $\sigma=1.2$ & $\sigma=1.3$ & $\sigma=1.4$   \\
		\noalign{\smallskip}\hline\noalign{\smallskip}
		$\beta=0.5$ & 93.16  & 92.9 & 93.78 & 93.32 & \textbf{94.09} & 93.38 & 93.54 & 93.62 & 93.23 & 93.67 \\
		$\beta=0.6$ & 93.19  & 93.14 & 93.78 & 93.58 & 93.87 & 93.56 & 93.25 & 93.43 & 93.19 & 93.19 \\
		$\beta=0.7$ & 93.32  & 92.97 & 93.73 & 93.47 & 93.91 & 93.38 & 93.52 & 93.69 & 92.86 & 93.65 \\
		$\beta=0.8$ & 93.27  & 93.73 & 93.56 & 93.32 & 93.5 & 92.64 & 92.94 & 93.41 & 93.32 & 93.73 \\
		$\beta=0.9$ & 93.36  & 93.62 & 93.34 & 93.43 & 93.45 & 91.78 & 93.16 & 93.89 & 93.47 & 93.36 \\
		$\beta=1$ & 93.69  & 93.56 & 93.76 & 93.41 & 93.45 & 90.83 & 92.86 & 93.58 & 93.45 & 93.38 \\
		$\beta=1.1$ & 93.58  & 93.52 & 93.23 & 93.25 & 93.34 & 91.65 & 92.75 & 93.73 & 93.25 & 93.19 \\
		$\beta=1.2$ & 93.56  & 93.45 & 93.23 & 93.38 & 93.43 & 92.75 & 93.08 & 94.08 & 93.41 & 93.21 \\
		$\beta=1.3$ & 93.67  & 93.16 & 93.25 & 93.32 & 93.47 & 92.5 & 92.81 & 94.06 & 92.57 & 93.23 \\
		$\beta=1.4$ & 93.36  & 93.23 & 93.43 & 93.95 & 93.71 & 93.34 & 93.34 & 93.8 & 92.64 & 93.47 \\
		\noalign{\smallskip}\hline
	\end{tabular}
\end{table*}

\section{Conclusion} \label{sec6}
This paper proposes a new graph linear canonical transform (GLCT). Furthermore, we discuss the vertex-frequency analysis framework based on this GLCT, delving into the specific delineations of different vertex-frequency analysis methods defined by various types of filters. Lastly, we examine filter design, encompassing optimal design and filter learning based on stochastic gradient descent, and apply them in image classification. The GLCT proposed in this paper demonstrates promising application prospects, with the filter design section akin to the convolutional layers of spectral graph neural networks (SGNNs), which can potentially be employed in large-scale SGNNs in the future. Subsequently, we will delve deeper into exploring the impact of linear canonical parameters to ensure their widespread applicability in practice.

\section*{Acknowledgments}
This work was funded by the National Natural Science Foundation of China [No. 62171041], and the Natural Science Foundation of Beijing Municipality [No. 4242011].

\bibliographystyle{IEEEtran}
\bibliography{mybib}

\vfill

\end{document}